\documentclass{amsart}
\usepackage{mathrsfs}
\usepackage{epsf, graphicx}
\usepackage{latexsym,amsfonts,amsbsy,amssymb}
\usepackage{subfigure}
\usepackage{multirow}
\usepackage{hyperref}
\usepackage{color}
\usepackage{lipsum}
\usepackage{amsfonts}
\usepackage{graphicx}
\usepackage{epstopdf}
\usepackage{color}
\usepackage{subfigure}
\usepackage{amssymb}
\usepackage{amsmath}
\usepackage{multirow}
\usepackage{enumerate}
\usepackage{cite,stmaryrd,txfonts}
\usepackage{tikz}
\usepackage{algorithm}
\usepackage{verbatim}
\usepackage{epic}
\usepackage{empheq}
\usepackage{ulem}

\usepackage{amsmath,amssymb,amsfonts}
\usepackage[T1]{fontenc}
\usepackage[english]{babel}
\usepackage{xspace}
\usepackage{amscd}
\usepackage{color}

\def\bw{\boldsymbol{w}}
\def\bv{\boldsymbol{v}}
\def\dv{\rm div}
\def\curl{\rm curl}
\def\rot{\rm rot}

\usepackage{amsmath,amssymb,mathrsfs}
\usepackage{graphicx}
\usepackage{stmaryrd}

\newtheorem{theorem}{Theorem}
\newtheorem{remark}[theorem]{Remark}

\newtheorem{lemma}[theorem]{Lemma}

\begin{document}
\title{A simple low-degree optimal finite element scheme for the elastic transmission eigenvalue problem}




\author{Yingxia Xi}
\address{School of Science, Nanjing University of Science and Technology, Nanjing 210094, China.}
\email{xiyingxia@njust.edu.cn}
\thanks{The research of Y. Xi is supported in part by the National Natural Science Foundation of China with Grant No.11901295, Natural Science Foundation of Jiangsu Province under BK20190431.}
\author{Xia Ji}
\address{School of Mathematics and Statistics, Beijing Institute of Technology, Beijing 100081, China.\\ Beijing Key Laboratory on MCAACI, Beijing Institute of Technology, Beijing 100081, China.}
\email{jixia@bit.edu.cn}
\thanks{The research of X. Ji is partially supported by the National Natural Science Foundation of China with Grant Nos.11971468 and 91630313.}
\author{Shuo Zhang}
\address{LSEC, Institute of Computational Mathematics and Scientific/Engineering Computing, Academy of Mathematics and System Sciences, Chinese Academy of Sciences, Beijing 100190, People's Republic of China}
\email{szhang@lsec.cc.ac.cn}
\thanks{The research of S. Zhang is supported partially by the National Natural Science Foundation of China with Grant Nos.11471026 and 11871465, the Strategic Priority Research Program of CAS, Grand No. XDB 41000000 and National Centre for Mathematics and Interdisciplinary Sciences, Chinese Academy of Sciences.}


\begin{abstract}
The paper presents a finite element scheme for the elastic transmission eigenvalue problem written as a fourth order eigenvalue problem. The scheme uses piecewise cubic polynomials and obtains optimal convergence rate. Compared with other low-degree and nonconforming finite element schemes, the scheme inherits the continuous bilinear form which does not need extra stabilizations and is thus simple to implement.
\end{abstract}

\maketitle

\section{Introduction}
The transmission eigenvalue problem is important in the qualitative reconstruction in the inverse scattering theory of inhomogeneous media. For example, the eigenvalues can be used to obtain estimates for the physical characteristics of the hidden scatterer \cite{CakoniEtal2010IP, Sun2011IP} and have a multitude of applications in inverse problems for target identification and nondestructive testing \cite{CakoniRen,HarrisCakoniSun}. Besides, the transmission eigenvalues play a key role in the uniqueness and reconstruction in inverse scattering theory. Moreover, they can be used to design invisibility materials \cite{JiLiu2016}. There are different types of transmission eigenvalue problems, such as the acoustic transmission eigenvalue problem, the electromagnetic transmission eigenvalue problem,  and the elastic transmission eigenvalue problem, etc.

Since 2010, effective numerical methods for the acoustic transmission eigenvalues have been developed by many researchers \cite{ColtonMonkSun2010, Sun2011SIAMNA,JiSunTurner2012ACMTOM, AnShen2013JSC, Geng2016, Kleefeld2013IP, JiSunXie2014JSC,CakoniMonkSun2014CMAM, LiEtal2014JSC, YangBiLiHan2016,
Camano2018,HuangEtal2016JCP,HuangEtal2017arXiv,SunZhou2016,XiJiZhang2018,XiJiZhang2020}, while there are much fewer works for the electromagnetic transmission eigenvalue problem and the elastic transmission eigenvalue problem \cite{HuangHuangLin2015SIAMSC,MonkSun2012SIAMSC,SunXu2013IP,JiLiSun,YangHanBi2020,XiJi2018}.
In this paper, we try to develop effective numerical methods for transmission eigenvalue problem of elastic waves.

The non-self-adjointness and nonlinearity plaguing the numerical study of the elastic transmission eigenvalue problem are compounded by the tensorial structure of the elastic wave equation. In \cite{JiLiSun}, the elastic transmission eigenvalue problem is reformulated as solving the roots of a nonlinear function of which the values correspond to the generalized eigenvalues of a series of fourth order self-adjoint eigenvalue problems discretized by $H^2$-conforming finite element methods. The authors apply the secant iterative method to compute the transmission eigenvalues. However, at each step, a fourth order self-adjoint eigenvalue problem needs to be solved and only real eigenvalues can be captured.
Based on a fourth order variational formulation, Yang et al. \cite{YangHanBi2020} study the $H^2$-conforming methods including the classical $H^2$-conforming finite element method and the $H^2$-conforming spectral element method.
In \cite{XiJi2018}, Xi et al. propose an interior penalty discontinuous Galerkin method using C$^0$ Lagrange elements (C$^0$IP) for the elastic transmission eigenvalue problem which use less degrees of freedom than $C^1$ elements and much easier to be implemented. 
There have also existed some mixed methods for this problem\cite{XiJi2020, YangHanBiLiZhang2020}.
The mixed scheme in \cite{XiJi2020}  has similarity to Ciarlet-Raviart discretization of biharmonic problem \cite{Ciarlet1974} and is one of the least expensive methods in terms of the amount of degrees of freedom. The convergence analysis is presented under the framework of spectral approximation theory of compact operator \cite{BabuskaOsborn1991,Osborn1975MC} and the error analysis of a mixed finite element method for the Stokes problem \cite{Girault1986}. Although the existence of elastic transmission eigenvalues
is beyond our concern, we want to remark that there exist only a few studies on the existence of the elasticity transmission eigenvalues \cite{Charalambopoulos2002JE,  CharalambopoulosAnagnostopoulos2002JE,
BellisGuzina2010JE, BellisCakoniGuzina}. We hope that the numerical results can give some hints on the analysis of the elasticity transmission eigenvalue
problem.

In this paper, we study the $H^2$-nonconforming finite element method $B_{h0}^3$ for the elastic transmission eigenvalue problem. The method is first introduced in \cite{zhang2018} for the biharmonic equation with optimal convergence rate. The method does not correspond to a locally defined finite element with Ciarlet's triple but admits a set of local basis functions. It is applied in the Helmholtz transmission eigenvalue problem \cite{XiJiZhang2020} obtaining the optimal convergence rate. For the elastic transmission eigenvalue problem,  the tensorial structure is the main difficulty.  The connection between bi-elastic eigenvalue problem and biharmonic eigenvalue problem is established which make the theoretical analysis possible. The property can be inherited by the $B_{h0}^3$ element at the discrete level, and we design schemes accordingly.
In this paper, by a careful analysis, we reveal that when $\lambda$ is not big, the bi-elastic problem takes the same essence as that of a tensorial biharmonic equation. This observation can hint a stabilization formulation for general low-degree finite element such as the Morley element. We give a new Morley element in section 5.3 and the idea can be applied to the other low-degree finite element.

The rest of this paper is organized as follows. In section 2, we introduce the problem and a piecewise cubic finite element. Section 3 presents the numerical scheme and error estimate for the bi-elastic problem. Section 4 gives the error estimate of the elasticity transmission eigenvalue problem. Numerical examples are presented in the last section. We also compare our method with Morley element scheme.
\section{Preliminaries}
\subsection{Elastic transmission eigenvalue problem}
Before introducing the elastic transmission eigenvalue problem, we present some notations first. Let $\Omega\subset\mathbb{R}^2$ be a bounded convex Lipschitz domain and $\boldsymbol x=(x_1, x_2)^\top\in\mathbb R^2$. And all vectors will be denoted in bold script in the subsequent. We denote the displacement vector of the wave field by $\boldsymbol{u}(\boldsymbol{x})=(u_1(\boldsymbol{x}), u_2(\boldsymbol{x}))^\top$ and the displacement gradient tensor by
\[
 \nabla\boldsymbol{u}=\begin{bmatrix}
                       \partial_{x_1} u_1 & \partial_{x_2} u_1\\
                       \partial_{x_1} u_2 & \partial_{x_2} u_2
                      \end{bmatrix}.
\] The strain tensor $\varepsilon(\boldsymbol{u})$ is given by
\[
\varepsilon(\boldsymbol{u})=\frac{1}{2}(\nabla\boldsymbol{u}
+(\nabla\boldsymbol{u})^\top),
\]
and the stress tensor $\sigma(\boldsymbol{u})$ is given by the generalized Hooke law
\[
 \sigma(\boldsymbol{u})=2\mu\varepsilon(\boldsymbol{u})+\lambda {\rm
tr}(\varepsilon(\boldsymbol{u})){\rm I},
\]
where the Lam\'{e} parameters $\mu,\  \lambda$ are two positive constants and ${\rm I}\in\mathbb{R}^{2\times 2}$ is the identity
matrix. Writing the above equation out, we have
\begin{equation}\label{DefSigma}
 \sigma(\boldsymbol{u})=\begin{bmatrix}
  (\lambda+2\mu)\partial_{x_1} u_1 + \lambda \partial_{x_2} u_2 & \mu (\partial_{x_2} u_1 +
\partial_{x_1} u_2)\\
\mu (\partial_{x_1} u_2 + \partial_{x_2} u_1) & \lambda\partial_{x_1} u_1 +
(\lambda+2\mu)\partial_{x_2} u_2
 \end{bmatrix}.
\end{equation}

The two-dimensional elastic wave problem is described by the reduced Navier equation: Find
$\boldsymbol{u}$ with zero trace on $\partial \Omega$, such that
\begin{equation}\label{ElasticityLHS1}
\nabla\cdot\sigma(\boldsymbol{u})+\omega^2 \rho \boldsymbol{u}={\boldsymbol 0},
\quad\text{in}~\Omega \subset \mathbb{R}^2,
\end{equation}
where $\omega>0$ is the angular frequency and $\rho$ is the mass density.

The elasticity transmission eigenvalue problem is to find $\omega^2\neq0$ such that there exist non-trivial solutions $\boldsymbol{u}, \boldsymbol{v}$ satisfying
\begin{equation}\label{tep}
\left\{
\begin{array}{rcl}
\nabla\cdot\sigma(\boldsymbol{u})+\omega^2 \rho_0 \boldsymbol{u}={\boldsymbol 0} &\quad\text{in} ~
\Omega,\\
\nabla\cdot\sigma(\boldsymbol{v})+\omega^2 \rho_1 \boldsymbol{v}={\boldsymbol 0} &\quad\text{in} ~
\Omega,\\
\boldsymbol{u}=\boldsymbol{v} &\quad\text{on} ~ \partial \Omega,\\
\sigma(\boldsymbol{u}){\boldsymbol n}=\sigma(\boldsymbol{v}){\boldsymbol n} &\quad\text{on} ~
\partial \Omega,
\end{array}
\right.
\end{equation}
where $\sigma {\boldsymbol n}$ denotes the matrix multiplication of the
stress tensor $\sigma$ and the unit outward normal $\boldsymbol n$. We assume that the mass density distributions satisfy the following inequalities
\begin{equation}\label{pP}
q \le \rho_0({\boldsymbol x}) \le Q, \quad q_* \le
\rho_1({\boldsymbol x}) \le Q_*, \quad {\boldsymbol x} \in \Omega,
\end{equation}
where $q, q_*$ and $Q, Q_*$ are positive constants.  We also assume that the two density distributions are "non-intersecting"\cite{BellisCakoniGuzina}, i.e.
\begin{equation}
Q\leq1\leq q_*  \ \ \ \ or \ \ \ \ Q_*\leq1\leq q  \nonumber.
\end{equation}

Define the Sobolev space
\begin{equation}\label{VVVV}
V=\{\boldsymbol{\phi} \in (H^2(\Omega))^2: {\boldsymbol \phi} = {\boldsymbol 0} \text{
and } \sigma({\boldsymbol \phi}) {\boldsymbol \nu} = {\bf 0} ~ \text{on}~ \partial \Omega \}.
\end{equation}
Introducing a new variable $\boldsymbol{w}=\boldsymbol{u}-\boldsymbol{v}\in V$, following the same procedure in \cite{JiLiSun}, the system (\ref{tep})
can be written as follows
\begin{equation}
(\nabla\cdot\sigma+\omega^2\rho_1)(\rho_1-\rho_0)^{-1}(\nabla\cdot\sigma+\omega^2\rho_0)\boldsymbol{w}=\boldsymbol{0}.
\end{equation}
Let $\tau=\omega^2$. The corresponding weak formulation is to find $\tau$ and $\boldsymbol{w}\neq0\in V$ such that
\begin{equation}\label{WeakFormula}
\left((\rho_1-\rho_0)^{-1}(\nabla\cdot\sigma+\tau\rho_0)\boldsymbol{w},(\nabla\cdot\sigma+\tau\rho_1)\boldsymbol{\varphi}\right)=\boldsymbol{0},~~~~~\forall\boldsymbol{\varphi}\in V.
\end{equation}


%
%
\subsection{A piecewise cubic finite element}
Before introducing the finite element, we give some notations.
Assume $\mathcal{T}_h$ a shape regular mesh over $\Omega$ with mesh size $h$. Denote $\mathcal{X}_h$, $\mathcal{X}_h^i$, $\mathcal{X}_h^b$, $\mathcal{E}_h$, $\mathcal{E}_h^i$, $\mathcal{E}_h^b$ the vertices, interior vertices, boundary vertices, the set of edges, interior edges and boundary edges, respectively. For any edge $e\in \mathcal{E}_h$, denote the unit normal vector of $e$ by $\bf{n}_e$ and denote the unit tangential vector of $e$ by $\bf{t}_e$.
For a triangle $T\in\mathcal{T}_h$, we use $\mathcal{P}_k(T)$ to denote the set of polynomials not higher than k and $|T|$ means the area measurement of element $T$. On an edge $e$, $\mathcal{P}_k(e)$ and $|e|$ are defined similarly.

The cubic $H^2$-nonconforming finite element space $B_h^3$ \cite{zhang2018} can be defined
as follows:
\begin{eqnarray*}
B_h^3&=&\big\{v\in L^2(\Omega)\mid\ v|_T\in \mathcal{P}_3(T),\ v\ {\rm is\ continuous\ at\ vertices}\ a\in\mathcal{X}_h  \rm{\ and}\nonumber\\
&&\ \int_{e}\llbracket v\rrbracket\ ds =0,\ {\rm  and}\ \int_{e}p_e\llbracket \partial_n v \rrbracket\ ds =0,\ \forall\ p_e\in P_1(e),\ \forall\ e\in \mathcal{E}_h^i,\ \forall\ T\in\mathcal{T}_h\big\},
\end{eqnarray*}
where $\llbracket v\rrbracket$ represents the jump of the scalar function $v$ across e, and
\begin{eqnarray*} 
B_{h0}^3&=&\big\{v\in B_h^3\mid\ v(a)=0, \ a\in \mathcal{X}_h^{b};\ \int_{e} v\ ds =0,\ {\rm  and}\ \int_{e}p_e \partial_n v \ ds =0,\ \forall\ p_e\in P_1(e),\ \forall\ e\in \mathcal{E}_h^b\}.
\end{eqnarray*}
It should be noticed that $B^3_{h0}$
does not correspond to a locally defined finite element with Ciarlet$'$s triple but admits a set of
local basis functions. The explicit expression of basis functions can be referred to \cite{XiJiZhang2020}.

Besides, we denote some finite element spaces which will be used in the subsequent by
\begin{eqnarray}\nonumber
&V_{h}^k&=\left\{w\in H^1(\Omega): w|_{T}\in P_k(T),\ \forall\ T\in\mathcal{T}_h\right\},\ \ \ V_{h0}^k=V_{h}^k\cap H_0^1(\Omega),\ k\geqslant 1;\\ \nonumber
&P_h^k&=\left\{w\in L^2(\Omega): w|_{T}\in P_k(T),\ \forall\ T\in\mathcal{T}_h\right\},\ \ \ P_{h0}^k=P_{h}^k\cap L_0^2(\Omega),\ k\geqslant 0;\\ \nonumber
&(G_h^k)^2&=\left\{\boldsymbol{v}\in (L^2(\Omega))^2: \int_e p_e\llbracket v_l\rrbracket ds=0,\ \ \ \forall p_e\in P_{k-1}(e),\ \forall e\in \mathcal{E}_h^i,\ l=1,2\right\},\ k\geqslant 1; \\ \nonumber
&(G_{h0}^k)^2&=\left\{\boldsymbol{v}\in (G_h^k)^2: \int_e p_e v_l ds=0,\ \ \ \forall p_e\in P_{k-1}(e),\ \forall e\in \mathcal{E}_h^b,\ l=1,2\right\},\ k\geqslant 1.\\  \nonumber
&(S^k_h)^2&=(P_h^k)^2 \cap (H^1(\Omega))^2,\  (S^k_{h0})^2=(S^k_h)^2 \cap (H_0^1(\Omega))^2,\  k\geqslant 1.\\ \nonumber
&B_{h0}^2&=\left\{\phi_h:{ \phi_h}|_{T}\in \mbox{span}\{\lambda_1^2+\lambda_2^2+\lambda_3^2-2/3\},\ \forall\ T\in\mathcal{T}_h\right\},\nonumber
\end{eqnarray}
where $\lambda_i, \ i=1,2,3$ are the barycentric coordinates.

For the $B_{h0}^3$ space, we can verify the following results.
\begin{lemma}\cite{zhang2019}
For $\forall w_h,v_h\in B_{h0}^3$, we have $(\Delta_hw_h,\Delta_hv_h)=(\nabla_h^2w_h,\nabla_h^2v_h)$.
\end{lemma}
\begin{lemma}\label{lemma2.2}
For $\forall \bw{}_h,\bv{}_h\in (B_{h0}^3)^2$, it can be established that $(\nabla_h\dv_h\bw{}_h,\curl_h\rot_h\bv{}_h)=0$.
\end{lemma}
\begin{proof}
First, if $w_h, v_h\in G_{h0}^k$, we can get
\begin{equation}(\nabla_h w_h,\curl_h v_h)=0.
\label{lemm221} \end{equation}
Since $G_{h0}^k=S^k_{h0}\oplus B_{h0}^2$ \cite{zhang2019}, $v{}_h$ can be written as
$$v{}_h=v{}_h^1+v{}_h^2, \ v{}_h^1\in S^k_{h0}, \ v{}_h^2 \in B_{h0}^2.$$
For $v{}_h^1$,
\begin{eqnarray}\label{lemm222}
(\nabla_h w{}_h, \curl_h v{}_h^1)&=&\sum_T (\nabla w{}_h, \curl v{}_h^1)_T=\sum_T \int_{\partial T} w{}_h \partial_{\bf{t}} v{}_h^1 \\ \nonumber
&=& \sum_e \int_e \llbracket w{}_h \rrbracket \partial_{\bf{t}_e} v{}_h^1=0.
\end{eqnarray}
here we use $\partial_{\bf{t}_e} v{}_h^1\in P_1(e)$ and $\int_e p_e \llbracket w{}_h \rrbracket=0, \ \forall p_e \in P_1(e)$.

For $v{}_h^2$,
\begin{eqnarray}\label{lemm223}
(\nabla_h w{}_h, \curl_h v{}_h^2)&=&\sum_T (\nabla w{}_h, \curl v{}_h^2)_T=\sum_T \int_{\partial T} \partial_{\bf{t}} w{}_h  v{}_h^2=0,
\end{eqnarray}
since $\partial_{\bf{t}} w{}_h|_{\partial T}  \in P_1(\partial _T)$ and $ v{}_h^2$ is bubble on the boudary.

The combination \eqref{lemm222} with \eqref{lemm223} leads to \eqref{lemm221}.

If $\bw{}_h\in (B_{h0}^3)^2$, then $\dv_h\bw{}_h, \ \rot_h\bv{}_h \in G_{h0}^k$,  the proof is complete.

\end{proof}

\section{A second order scheme for the bi-elastic problem}
Based on the above equation (\ref{ElasticityLHS1}), we construct the bi-elastic eigenvalue problem: Find $(K,\boldsymbol{w})\in \mathcal{C}\times V$ satisfying
\begin{equation}\label{ElasticEigenvalueProblem}
\left\{
\begin{array}{rcl}
\left(\nabla\cdot\sigma\right)(\beta\nabla\cdot\sigma(\boldsymbol{w}))&=&K\boldsymbol{w},\ \ \ \ \rm{in}\ \ \Omega,\\
{\boldsymbol w}&=&{\bf 0},\ \ \ \ \rm{on}\ \ \partial \Omega,\\
\sigma({\boldsymbol w}) {\boldsymbol \nu}&=&{\bf 0},\ \ \ \ \rm{on}\ \ \partial \Omega,
\end{array}
\right.
\end{equation}
where $\sigma {\boldsymbol \nu}$ denotes the matrix multiplication of the
stress tensor $\sigma$ and the unit outward normal $\boldsymbol \nu$. And $\beta(\boldsymbol{x})$ is a bounded smooth non-constant function, $\beta(\boldsymbol{x})\geqslant \beta_{\min}>0$.

The corresponding variational formulation is to find $(K, \boldsymbol{w})\in \mathcal{C}\times V$, such that
\begin{equation}\label{AuxiliaryVariational}
A_{\beta}(\boldsymbol{w},\boldsymbol{\varphi})\triangleq\left(\beta\nabla\cdot\sigma(\boldsymbol{w}),\nabla\cdot\sigma(\boldsymbol{\varphi})\right)=K(\boldsymbol{w},\boldsymbol{\varphi}),\ \ \ \forall \boldsymbol{\varphi}\in V.
\end{equation}
\begin{theorem}\label{CCoerciveness}
$\|\nabla\cdot\sigma(\boldsymbol{w})\|_0\cong\|\boldsymbol{w}\|_2,~\forall\boldsymbol{w}\in V$. ("$\cong$" denotes "$=$" up to a constant.)
\end{theorem}
\begin{proof}
Especially, we can obtain that
\begin{small}
\begin{eqnarray}\label{Elastic_Derivations}
\Big(\nabla\cdot\sigma(\boldsymbol{w}),\nabla\cdot\sigma(\boldsymbol{\varphi})\Big)&=&\Big(-\mu\Delta\boldsymbol{w}-\lambda\nabla\nabla\cdot(\boldsymbol{w}),-\mu\Delta\boldsymbol{\varphi}-\lambda\nabla\nabla\cdot(\boldsymbol{\varphi})\Big)\nonumber\\
&=&\Big(-(\lambda+\mu)\nabla\nabla\cdot\boldsymbol{w}+\mu\ curl\ rot\boldsymbol{w}, -(\lambda+\mu)\nabla\nabla\cdot\boldsymbol{\varphi}+\mu\ curl\ rot\boldsymbol{\varphi}\Big)\nonumber\\
&=&(\lambda+\mu)^2\Big(-\nabla\nabla\cdot\boldsymbol{w}, -\nabla\nabla\cdot\boldsymbol{\varphi}\Big)+\mu^2\Big(curl\ rot\boldsymbol{w},curl\ rot\boldsymbol{\varphi}\Big) \nonumber\\
&=&\mu^2(-\Delta\boldsymbol{w}, -\Delta\boldsymbol{\varphi})+[(\lambda+\mu)^2-\mu^2](-\nabla\nabla\cdot(\boldsymbol{w}), -\nabla\nabla\cdot(\boldsymbol{\varphi})) , ~\forall \boldsymbol{w},\boldsymbol{\varphi}\in V,
\end{eqnarray}
\end{small}
then we get
\begin{small}
\begin{eqnarray}
\mu^2\| -\Delta\boldsymbol{w}\|_0^2\leq \|\nabla\cdot\sigma(\boldsymbol{w}) \|_0^2  \leq (\lambda+\mu)^2\| -\Delta\boldsymbol{w}\|_0^2.
\end{eqnarray}
\end{small}
Further, combining the above equation with $\| -\Delta\boldsymbol{w}\|_0^2=\|\nabla^2\boldsymbol{w} \|_0, \boldsymbol{w} \in (H^2\cap H_0^1)^2  $ leads to the desired result.
\end{proof}
\begin{remark}
From the above derivations, we know that $A_\beta(\cdot,\cdot)$ is coercive on $V$. And it's easy to verify the boundedness of $A_\beta(\cdot,\cdot)$. The well-posedness of (\ref{AuxiliaryVariational}) follows by Lax-Milgram theorem.
\end{remark}

\begin{remark}
 The above theorem establishes the connection between bi-elastic eigenvalue problem and biharmonic eigenvalue problem. The theoretical analysis can be put down to the biharmonic eigenvalue problem. We give some details to make the paper more understandable.
\end{remark}

The bi-elastic source problem for (\ref{ElasticEigenvalueProblem}) is to find $\boldsymbol{w}\in V$ satisfying
\begin{equation}\label{ElasticBoundaryProblem}
\left\{
\begin{array}{rcl}
\left(\nabla\cdot\sigma\right)(\beta\nabla\cdot\sigma(\boldsymbol{w}))&=&\boldsymbol{f},\ \ \ \ \rm{in}\ \ \Omega,\\
{\boldsymbol w}&=&{\bf 0},\ \ \ \ \rm{on}\ \ \partial \Omega,\\
\sigma({\boldsymbol w}) {\boldsymbol \nu}&=&{\bf 0},\ \ \ \ \rm{on}\ \ \partial \Omega.
\end{array}
\right.
\end{equation}
The cubic finite element scheme for \eqref{ElasticBoundaryProblem} is defined as: find $\boldsymbol{w_h} \in (B^3_{h0})^2$, such that
\begin{equation}\label{Non-constantBoundaryValueProblemdis}
A_{\beta,h}(\boldsymbol{w_h},\boldsymbol{\varphi_h})\triangleq(\beta\left(\nabla\cdot\sigma\right)\boldsymbol{w_h},\left(\nabla\cdot\sigma\right)\boldsymbol{\varphi_h})=(\boldsymbol{f},\boldsymbol{\varphi_h}),\quad\forall\ \boldsymbol{\varphi_h}\in (B^3_{h0})^2.
\end{equation}
Define $\|\boldsymbol{w}_h\|_h=\left\{\sum\limits_{T\in\mathcal{T}_h}\int_T|\nabla\cdot\sigma(\boldsymbol{w}_h)|^2 dx_1dx_2\right\}^{\frac{1}{2}}$. In the following, we prove that $\|\boldsymbol{w}_h\|_h$ is a norm on $(B^3_{h0})^2$.

\begin{lemma}
$\|\boldsymbol{w}_h\|_h$ is a norm on $(B^3_{h0})^2$.
\end{lemma}
\begin{proof}
Using Lemma \ref{lemma2.2}, we obtain
 \begin{small}
\begin{eqnarray}\label{Elastic_Derivationsdis}
\|\boldsymbol{w}_h\|_h^2&=&\sum_{T\in\mathcal{T}_h}\Big(\nabla\cdot\sigma(\boldsymbol{w}_h), \nabla\cdot\sigma(\boldsymbol{w}_h)\Big)\nonumber\\
&=&\sum_{T\in\mathcal{T}_h}\Big(-(\lambda+\mu)\nabla\nabla\cdot\boldsymbol{w_h}+\mu\ curl\ rot\boldsymbol{w_h}, -(\lambda+\mu)\nabla\nabla\cdot\boldsymbol{w_h}+\mu\ curl\ rot\boldsymbol{w_h}\Big)\nonumber\\
&=&\sum_{T\in\mathcal{T}_h}(\lambda+\mu)^2\Big(-\nabla\nabla\cdot\boldsymbol{w_h}, -\nabla\nabla\cdot\boldsymbol{w_h}\Big)+\sum_{T\in\mathcal{T}_h}\mu^2\Big(curl\ rot\boldsymbol{w_h},curl\ rot\boldsymbol{w_h}\Big) \nonumber\\
&=&\sum_{T\in\mathcal{T}_h}(\mu^2(-\Delta\boldsymbol{w_h}, -\Delta\boldsymbol{w_h})+\sum_{T\in\mathcal{T}_h}([(\lambda+\mu)^2-\mu^2](-\nabla\nabla\cdot(\boldsymbol{w_h}), -\nabla\nabla\cdot(\boldsymbol{w_h})) \nonumber\\
&\geq&\mu^2\| \Delta_h \boldsymbol{w_h} \|_0^2=\mu^2\| \nabla_h^2 \boldsymbol{w_h} \|_0^2
, ~\forall \boldsymbol{w_h}\in (B^3_{h0})^2.
\end{eqnarray}
\end{small}
\end{proof}

\begin{lemma}
The finite element approximation (\ref{Non-constantBoundaryValueProblemdis}) is well-posed.
\end{lemma}
\begin{proof}
By the definition of $\|\boldsymbol{w}_h\|_h$, it's easy to know that $A_{\beta,h}(\cdot,\cdot)$ is coercive and bounded. By the Lax-Milgram theorem, (\ref{Non-constantBoundaryValueProblemdis}) is well-posed. The proof is complete.
\end{proof}

\begin{lemma}(c.f.\cite{zhang2018})\label{AuxiliaryOperator}
For $\forall\ \boldsymbol{w}\in (H_0^2(\Omega)\cap H^k(\Omega))^2(k=3,4)$, there exists a positive constant $C$ such that
\begin{equation}
\inf_{\boldsymbol{v_h}\in (B^3_{h0})^2} |\boldsymbol{w}-\boldsymbol{v_h}|_{2,h}\leq ch^{k-2} |\boldsymbol{w}|_{k,\Omega}.
\end{equation}
\end{lemma}

\begin{lemma}\label{boundary1}
Assume $\mathcal{T}_h$ be a shape regular mesh over $\Omega$ with mesh size $h$. There exists a constant $C>0$ such that for $\forall \boldsymbol{w}\in(H_0^2(\Omega)\bigcap H^k(\Omega))^2\ (k=3,4)$, it holds that
\begin{equation}
\sum_{T\in\mathcal{T}_h}\int_{\partial T}\left(\sigma(\nabla\cdot\sigma(\boldsymbol{w}))\boldsymbol{n}\right)\cdot\boldsymbol{v_h}ds\leqslant C h^{k-2}|\boldsymbol{w}|_{k,\Omega}|\boldsymbol{v_h}|_{2,\Omega}. \nonumber
\end{equation}
\end{lemma}
\begin{proof}
 The proof follows the same idea in Lemma 15 of \cite{zhang2018}, we ignore the details.
\end{proof}

\begin{lemma}\label{boundary2}
Under the assumption of Lemma \ref{boundary1}, there exists a constant $C>0$ such that for $\forall \boldsymbol{w}\in(H_0^2(\Omega)\bigcap H^k(\Omega))^2\ (k=3,4)$, it holds that
\begin{equation}
\sum_{T\in\mathcal{T}_h}2\mu\int_{\partial T}(\nabla\cdot\sigma(\boldsymbol{w}))\cdot\left(\varepsilon(\boldsymbol{v_h})\cdot\boldsymbol{n}\right)ds+\sum_{T\in\mathcal{T}_h}\lambda\int_{\partial T}\left(\nabla\cdot\sigma(\boldsymbol{w})\right)\cdot\left(\nabla\cdot\boldsymbol{v_h}\right)\cdot\boldsymbol{n}ds\leqslant C h^{k-2}|\boldsymbol{w}|_{k,\Omega}|\boldsymbol{v_h}|_{2,T}. \nonumber
\end{equation}
\begin{proof}
 The proof follows the same idea in Lemma 15 and 16 of \cite{zhang2018}, we ignore the details.
\end{proof}
\end{lemma}

\begin{theorem}\label{theorem_Bh03}
Let $\boldsymbol{w}\in (H^{k}(\Omega))^2\cap V(k=3,4)$ be the solution of \eqref{ElasticBoundaryProblem}, and $\boldsymbol{w_h}$ be the solution of \eqref{Non-constantBoundaryValueProblemdis}, respectively. Then
$$
\|\nabla\cdot\sigma(\boldsymbol{w}-\boldsymbol{w_h})\|_{0,\Omega}\leqslant Ch^{k-2}|\boldsymbol{w}|_{k,\Omega}.
$$
\end{theorem}
\begin{proof}
By Strang lemma,
\begin{equation}
\|\nabla\cdot\sigma(\boldsymbol{w}-\boldsymbol{w_h})\|_{0,\Omega}\leqslant C\left(\inf\limits_{\boldsymbol{v_h}\in(B_{h0}^3)^2}\|\nabla\cdot\sigma(\boldsymbol{w}-\boldsymbol{w_h})\|_{0,\Omega}+\sup\limits_{\boldsymbol{v_h}\neq\boldsymbol{0}\in(B_{h0}^3)^2}\frac{\left(\nabla\cdot\sigma(\boldsymbol{w}),\nabla\cdot\sigma(\boldsymbol{v_h})\right)-(\boldsymbol{f},\boldsymbol{v_h})}{\|\nabla\cdot\sigma(\boldsymbol{v_h})\|_{0,D}} \right).
\end{equation}
The approximation error estimate follows by Lemma \ref{AuxiliaryOperator}. Next we consider the consistency error.
For $\forall \boldsymbol{v_h}\in(B_{h0}^3)^2$, by (\ref{ElasticBoundaryProblem}),
\begin{eqnarray}\label{OriginProblem}
(\boldsymbol{f},\boldsymbol{v_h})&=&(\nabla\cdot\sigma(\nabla\cdot\sigma(\boldsymbol{w})),\boldsymbol{v_h})=\Sigma_{T\in\mathcal{T}_h}\int_T\left(\nabla\cdot\sigma(\nabla\cdot\sigma(\boldsymbol{w}))\right)\cdot\boldsymbol{v_h}dxdy. \end{eqnarray}
We focus on an element $T\in\mathcal{T}_h$ and use the Green formula, then
\begin{eqnarray}\label{GreenFormula}\nonumber
&&\int_T\left(\nabla\cdot\sigma(\nabla\cdot\sigma(\boldsymbol{w}))\right)\cdot\boldsymbol{v_h}dxdy=\int_{\partial T}\left(\sigma(\nabla\cdot\sigma(\boldsymbol{w}))\boldsymbol{n}\right)\cdot\boldsymbol{v_h}ds-\int_T\sigma(\nabla\cdot\sigma(\boldsymbol{w})):\nabla\boldsymbol{v_h}dxdy\\ \nonumber
&=&\int_{\partial T}\left(\sigma(\nabla\cdot\sigma(\boldsymbol{w}))\boldsymbol{n}\right)\cdot\boldsymbol{v_h}ds-2\mu\int_{\partial T}(\nabla\cdot\sigma(\boldsymbol{w}))\cdot\left(\varepsilon(\boldsymbol{v_h})\cdot\boldsymbol{n}\right)ds-\lambda\int_{\partial T}\left(\nabla\cdot\sigma(\boldsymbol{w})\right)\left(\nabla\cdot\boldsymbol{v_h}\right)\cdot\boldsymbol{n}ds \\ 
&+&\int_{T}\left(\nabla\cdot\sigma(\boldsymbol{w})\right)\cdot\left(\nabla\cdot\sigma(\boldsymbol{v_h})\right)dxdy. 
\end{eqnarray}
Hence, the combination of (\ref{OriginProblem}), (\ref{GreenFormula}), Lemma \ref{boundary1} and Lemma \ref{boundary2} leads to the desired result.
\end{proof}

The finite element space $B^3_{h0}$ leads immediately to a high-accuracy scheme for the eigenvalue
problem of bi-elastic problem (\ref{ElasticEigenvalueProblem}).

\section{Error estimate of the elasticity transmission eigenvalue problem}\label{ETEP}
Here we apply the $B_{h0}^3$ scheme to the elastic transmission eigenvalue problem \eqref{WeakFormula}. To analyse the error estimate, we first define the sesquilinear forms on $V\times V$
\begin{eqnarray}
A_\tau(\boldsymbol{w},\boldsymbol{\varphi})&=&\left((\rho_1-\rho_0)^{-1}(\nabla\cdot\sigma+\tau\rho_0)\boldsymbol{w},(\nabla\cdot\sigma+\tau\rho_0)\boldsymbol{\varphi}\right)+\tau^2(\rho_0\boldsymbol{w},\boldsymbol{\varphi})\nonumber \\
\tilde{A}_\tau(\boldsymbol{w},\boldsymbol{\varphi})&=&\left((\rho_0-\rho_1)^{-1}(\nabla\cdot\sigma+\tau\rho_1)\boldsymbol{w},(\nabla\cdot\sigma+\tau\rho_1)\boldsymbol{\varphi}\right)+\tau^2(\rho_1\boldsymbol{w},\boldsymbol{\varphi})\nonumber\\ B(\boldsymbol{w},\boldsymbol{\varphi})&=&\left(\sigma(\boldsymbol{w}),\nabla\boldsymbol{\varphi}\right).\nonumber
\end{eqnarray}
It's easy to verify that $A_\tau(\boldsymbol{w},\boldsymbol{\varphi}),\ \tilde{A}_\tau(\boldsymbol{w},\boldsymbol{\varphi})$ and $B(\boldsymbol{w},\boldsymbol{\varphi})$ are symmetric.
Using the Green formula, the variational problem for (\ref{WeakFormula}) can be written as to find $\tau\in R$ and $\boldsymbol{w}\in V$ such that
\begin{eqnarray}\label{Weak_Formula}
A_\tau(\boldsymbol{w},\boldsymbol{\varphi})=\tau B(\boldsymbol{w},\boldsymbol{\varphi}),\ \ \ \ \forall\boldsymbol{\varphi}\in V.
\end{eqnarray}
for $Q\leq1\leq q_*$
and
\begin{eqnarray}
\tilde{A}_\tau(\boldsymbol{w},\boldsymbol{\varphi})=\tau B(\boldsymbol{w},\boldsymbol{\varphi}),\ \ \ \ \forall\boldsymbol{\varphi}\in V.
\end{eqnarray}
for $Q_*\leq1\leq q$.

\begin{lemma}\cite{BellisCakoniGuzina,JiLiSun}
Let $\rho_0(x),\ \rho_1(x)$ be smooth enough and assume that $Q\leq1\leq q_*$. Then, $A_\tau$ is a coercive sesquilinear form on $V\times V$, i.e., there exists a constant $\gamma>0$ such that
\begin{equation}
A_\tau(\boldsymbol{w},\boldsymbol{w})\geqslant \gamma\parallel\boldsymbol{w}\parallel^2,\\\\\\~~\forall \boldsymbol{w}\in V.
\end{equation}
Under the assumption $Q_*\leq1\leq q$, we can also obtain that
$\tilde{A}_\tau$ is a coercive sesquilinear form on $V\times V$, i.e., there exists a constant $\gamma>0$ such that
\begin{equation}
\tilde{A}_\tau(\boldsymbol{w},\boldsymbol{w})\geqslant \gamma\parallel\boldsymbol{w}\parallel^2,\\\\\\~~\forall \boldsymbol{w}\in V.
\end{equation}
Besides, the bilinear form $\mathcal{B}(\cdot,\cdot)$ is symmetric and nonnegative on $V\times V$.
\end{lemma}

In the following, we take the case $Q\leq1\leq q_*$ for illustration. For the case $Q_*\leq1\leq q$, it follows similarly.
For a fixed point $\tau$, we define the following generalized eigenvalue problem:
Find $(\lambda(\tau),\boldsymbol{w})\in \mathcal{R}\times V$ such that $B(\boldsymbol{w},\boldsymbol{w})=1$ and
\begin{eqnarray}\label{Eigenvalue_A_Tau}
A_{\tau}(\boldsymbol{w},\boldsymbol{v})&=&\lambda(\tau)B(\boldsymbol{w},\boldsymbol{v}),\ \ \ \ \forall \boldsymbol{v}\in V.
\end{eqnarray}
The original eigenvalue problem (\ref{WeakFormula}) can be converted to solve the zero point of the nonlinear function
\begin{equation}\label{continuous_nonlinear_form}
f(\tau)=\lambda(\tau)-\tau.
\end{equation}
Furthermore, from the definitions
of $A_{\tau}(\cdot,\cdot)$, $f(\tau)$ is continuous corresponding to $\tau$ based on the eigenvalue
perturbation theory (c.f. \cite{Babuska,BabuskaOsborn1991}) and it's easy to verify the existence of zero points (c.f.\cite{JiLiSun}).
In this paper, we consider the numerical method for (\ref{continuous_nonlinear_form}).

By $B_{h0}^3$ scheme, the corresponding discretization form is to find $(\tau_h,\boldsymbol{w_h})\in\mathcal{R}\times (B_{h0}^3)^2$ such that $\mathcal{B}(\boldsymbol{w_h},\boldsymbol{w_h})=1$ and
\begin{eqnarray}\label{Discrete_Eigenvalue_Problem_h}
A_{\tau_h,h}(\boldsymbol{w_h},\boldsymbol{\varphi_h}) &=& \tau_h B_h(\boldsymbol{w_h},\boldsymbol{\varphi_h}), \ \ \ \forall \boldsymbol{\varphi_h}\in (B_{h0}^3)^2.
\end{eqnarray}
Similar to (\ref{Eigenvalue_A_Tau}), we define the discretized generalized eigenvalue problem: Find $(\lambda_h(\tau),\boldsymbol{w_h})\in\mathcal{R}\times (B_{h0}^3)^2$
such that $B_h(\boldsymbol{w_h},\boldsymbol{w_h})=1$ and
\begin{eqnarray}\label{Auxiliary_Eigenvalue_Problem_h}
A_{\tau,h}(\boldsymbol{w_h},\boldsymbol{\varphi_h}) &=& \lambda_{h}(\tau)B_h(\boldsymbol{w_h},\boldsymbol{\varphi_h}),
 \ \ \ \forall \boldsymbol{\varphi_h}\in (B_{h0}^3)^2.
\end{eqnarray}
The corresponding discretized nonlinear function is
\begin{equation}\label{discrete_nonlinear_form}
f_h(\tau)=\lambda_h(\tau)-\tau.
\end{equation}
\begin{lemma}\label{ApproximationEvs1}
In terms of (\ref{Eigenvalue_A_Tau}) and (\ref{Auxiliary_Eigenvalue_Problem_h}), let $\boldsymbol{w}\in (H^{4}(\Omega))^2\cap V$, then we can obtain the following approximate result
\begin{equation}
|\lambda(\tau)-\lambda_h(\tau)|< C h^4.
\end{equation}
\end{lemma}
\begin{proof}
Based on the proof of theorem \ref{theorem_Bh03}, we can prove this lemma similarly.
\end{proof}

The following lemma states that the root of (\ref{discrete_nonlinear_form}) approximates the root of (\ref{continuous_nonlinear_form}) well if the mesh size is small enough.

\begin{lemma}(\cite{JiLiSun})\label{ApproximationEvs2}
Let $f(\tau)$ and $f_h(\tau)$ be two continuous functions. For small enough $\epsilon>0$, there exists some $\eta>0$ such that $f'(\tau)\leqslant-\eta<0$ and $|f(\tau)-f_h(\tau)|<\varepsilon$ on $[x_1-\frac{\epsilon}{\delta},x_2+\frac{\epsilon}{\delta}]$, for some $0<x_1<x_2$,
$\delta>0$ is constant. If there exists $\widetilde{\tau}\in[x_1,x_2]$ such that $f_h(\widetilde{\tau})=0$, then there exists a $\tau^\ast$ such that $f(\tau^\ast)=0$ and holds the following approximate formula
\begin{equation}
|\widetilde{\tau}-\tau^\ast|<\frac{\epsilon}{\delta}.
\end{equation}
\end{lemma}
Further, combining Lemma \ref{ApproximationEvs1} with Lemma \ref{ApproximationEvs2}, we can obtain
\begin{equation}
|\widetilde{\tau}-\tau^\ast|<C h^4,
\end{equation}
which implies that the convergence rate of elastic transmission eigenvalues using $B_{h0}^3$ is 4.

Assume that $\tau_h$ is the approximation of the exact eigenvalue $\tau$ satisfying (\ref{Weak_Formula}). Before giving the convergence analysis of the eigenfunction approximation, similar to the Helmholtz transmission eigenvalue in \cite{JiXiXie2017}, we introduce an auxiliary eigenvalue problem:
Find $(\tilde{\tau},\tilde{\boldsymbol{w}})\in\mathcal{R}\times V$
such that $B(\tilde{\boldsymbol{w}},\tilde{\boldsymbol{w}})=1$ and
\begin{eqnarray}\label{Auxiliary_Eigenvalue_Problem}
A_{\tau_h}(\tilde{\boldsymbol{w}},\boldsymbol{\varphi})&=&\tilde{\tau}B(\tilde{\boldsymbol{w}},\boldsymbol{\varphi}),\ \ \ \ \forall \boldsymbol{\varphi}\in V.
\end{eqnarray}
For (\ref{Weak_Formula}) and (\ref{Auxiliary_Eigenvalue_Problem}), using the standard theory of operator perturbation  \cite{JiXiXie2017, XiJiZhang2020}, we can obtain the following approximate results.
\begin{lemma}\label{Auxiliary2}
Assume that $\tau_h$ and $\tau$ satisfy $|\tau-\tau_h|<C h^4$. Let $(\tau,\boldsymbol{w}), (\tilde{\tau},\tilde{\boldsymbol{w}})$ be the exact solution of (\ref{Weak_Formula}) and (\ref{Auxiliary_Eigenvalue_Problem}), respectively. Let $\boldsymbol{w}\in (H^4(\Omega))^2\cap V$, then, we can obtain the following estimations

\begin{eqnarray}
\|\nabla\cdot\sigma(\boldsymbol{w}-\tilde{\boldsymbol{w}})\|_{0,\Omega}&\lesssim&h^3, \label{0} \\
|\tau-\tilde{\tau}|&\lesssim&h^4.
\end{eqnarray}

\end{lemma}

Using Lemma \ref{Auxiliary2} and the triangle inequality, we can obtain the following error estimates.
\begin{theorem}
Let $(\tau,\boldsymbol{w}), (\tau_h,\boldsymbol{w_h})$ be the solution of (\ref{Weak_Formula}) and (\ref{Discrete_Eigenvalue_Problem_h}), respectively. Let $\boldsymbol{w}\in (H^4(\Omega))^2\cap V$ and the domain $\Omega$ be convex. Then, there exist the following results
\begin{eqnarray}
\|\nabla\cdot\sigma(\boldsymbol{w}-\boldsymbol{w_h})\|_{0,\Omega}&\lesssim& h^2, \label{1}  \\
|\tau-\tau_h|&\lesssim& h^4. \label{3}
\end{eqnarray}
\end{theorem}

\begin{proof}
The estimation (\ref{3}) follows from Lemma \ref{ApproximationEvs2}. Actually, we can regard (\ref{Discrete_Eigenvalue_Problem_h}) as the discretized form of linear eigenvalue problem (\ref{Auxiliary_Eigenvalue_Problem}). Combining theorem \ref{theorem_Bh03} and the classical theory of nonconforming finite element method (c.f. \cite{BabuskaOsborn1991}), it's easy to verify
\begin{eqnarray}
\|\nabla\cdot\sigma(\tilde{\boldsymbol{w}}-\boldsymbol{w_h})\|_{0,\Omega}&\lesssim&h^2, \label{11}
\end{eqnarray}
Further, combining (\ref{0}) with (\ref{11}), we can obtain (\ref{1}).
\end{proof}

\section{Numerical Examples}\label{Experiments}
In this section, we give the numerical examples for the bi-elastic problem first, both source problem and eigenvalue problem are considered. What's more, the comparison with Morley element are presented. Then we discuss the numerical examples for the elastic transmission eigenvalue problem and also its comparison with Morley element. All examples are done using Matlab 2016a on a laptop with 16G memory and 2.9GHz Intel Core i7-7500U processor.

\subsection{The bi-elastic source problem}
In subsection 5.1 and 5.2, we use the initial mesh size $h_0=\frac{1}{2}$. Five levels of uniformly refined triangular meshes are generated for the numerical experiment and
$h_k=\frac{h_{k-1}}{2},\ k=1,2,3,4,5$, and $u_{h_k}$ represents the numerical solution on mesh $h_k$.

\textbf{Example 1}: We consider the boundary problem (\ref{ElasticBoundaryProblem}) on the unit square domain $\Omega_1=[0,1]^2$ with the constant coefficient $\beta(\boldsymbol{x})=1$ and Lam\'{e} parameters $\lambda=\frac{1}{4},\ \mu=\frac{1}{16}$. The right-hand item $\boldsymbol{f}$  in (\ref{ElasticBoundaryProblem}) is assumed as
\begin{eqnarray}\nonumber
f_1&=&\frac{3\pi^4sin(\pi x_2)}{256}(663cos(\pi x_1)^2cos(\pi x_2)^2-770cos(\pi x_1)cos(\pi x_2)+910cos(\pi x_1)^3cos(\pi x_2) \\ \nonumber
&-&347cos(\pi x_1)^2-345cos(\pi x_2)^2+177),\\ \nonumber
f_2&=&\frac{3\pi^4sin(\pi x_2)}{256}(663cos(\pi x_1)^2cos(\pi x_2)^2-770cos(\pi x_1)cos(\pi x_2)+910cos(\pi x_1)cos(\pi x_2)^3 \\ \nonumber
&-&345cos(\pi x_1)^2-347cos(\pi x_2)^2+177).
\end{eqnarray}
The exact solution is $\boldsymbol{w}=(sin(\pi x_1)^2 sin(\pi x_2)^3,sin(\pi x_1)^3 sin(\pi x_2)^2)^T$.

 The finest degrees of freedom on $h_5$ are 50182.
On each mesh level, we compute the error between the numerical solution and the exact solution measured by $L_2,\ H^1,\ H^2$ norms respectively. And the convergence rate is computed by
\begin{equation}
log_2\left(\frac{\|u-u_{h_{k}}\|}{\|u-u_{h_{k-1}}\|}\right),\ \ \ \,k=2,3,4,5,\nonumber
\end{equation}
 The numerical results are showed in the left graph of Figure \ref{BoundaryProblem}.
We can observe that

(1) The convergence rate for source problem measured by $H^2$ norm is $O(h^2)$;

(2) The convergence rate for source problem measured by $H^1$ norm is $O(h^3)$;

(3) The convergence rate for source problem measured by $L_2$ norm is $O(h^4)$;

which are optimal and consistent with the theoretical results.

\textbf{Example 2}: We consider (\ref{ElasticBoundaryProblem}) on the triangle domain $\Omega_2$ whose vertices are $(0,0), \ (1,0) \text{and} \ (0,1)$ with non-constant coefficient $\beta=8+x_1-x_2$ and Lam\'{e} parameters $\lambda=\frac{1}{4},\ \mu=\frac{1}{4}$. The right-hand item $\boldsymbol{f}$ is taken as
\begin{eqnarray}\nonumber
f_1&=&(49x_1^4)/2 + (289x_1^3x_2)/2 + 202x_1^3 + (123x_1^2x_2^2)/2 + 1080x_1^2x_2 - (345x_1^2)/2 - (149x_1x_2^3)/2+1308x_1x_2^2 \\ \nonumber
&-& (1425x_1x_2)/2 - 44x_2^4 + 450x_2^3 - 402x_2^2 + 108x_2,\\ \nonumber
f_2&=&44x_1^4 + (149x_1^3x_2)/2 + 358x_1^3 - (123x_1^2x_2^2)/2 + 1284x_1^2x_2 - 366x_1^2 - (289x_1x_2^3)/2 + 1296x_1x_2^2 \\ \nonumber
&-& (1551x_1x_2)/2 + 108x_1 - (49x_2^4)/2 + 230x_2^3 - (327x_2^2)/2.
\end{eqnarray}
The exact solution is $\boldsymbol{w}=(x_1^2x_2^3(x_1+x_2-1)^2,x_1^3x_2^2(x_1+x_2-1)^2)^T$.

The finest degrees of freedom on $h_5$ are 25350. The convergence behaviour is showed in the right graph of Figure \ref{BoundaryProblem}. The convergence orders are also optimal.
\begin{figure}
\centering
\subfigure{\includegraphics[width=0.45\textwidth,height=0.4\textwidth]{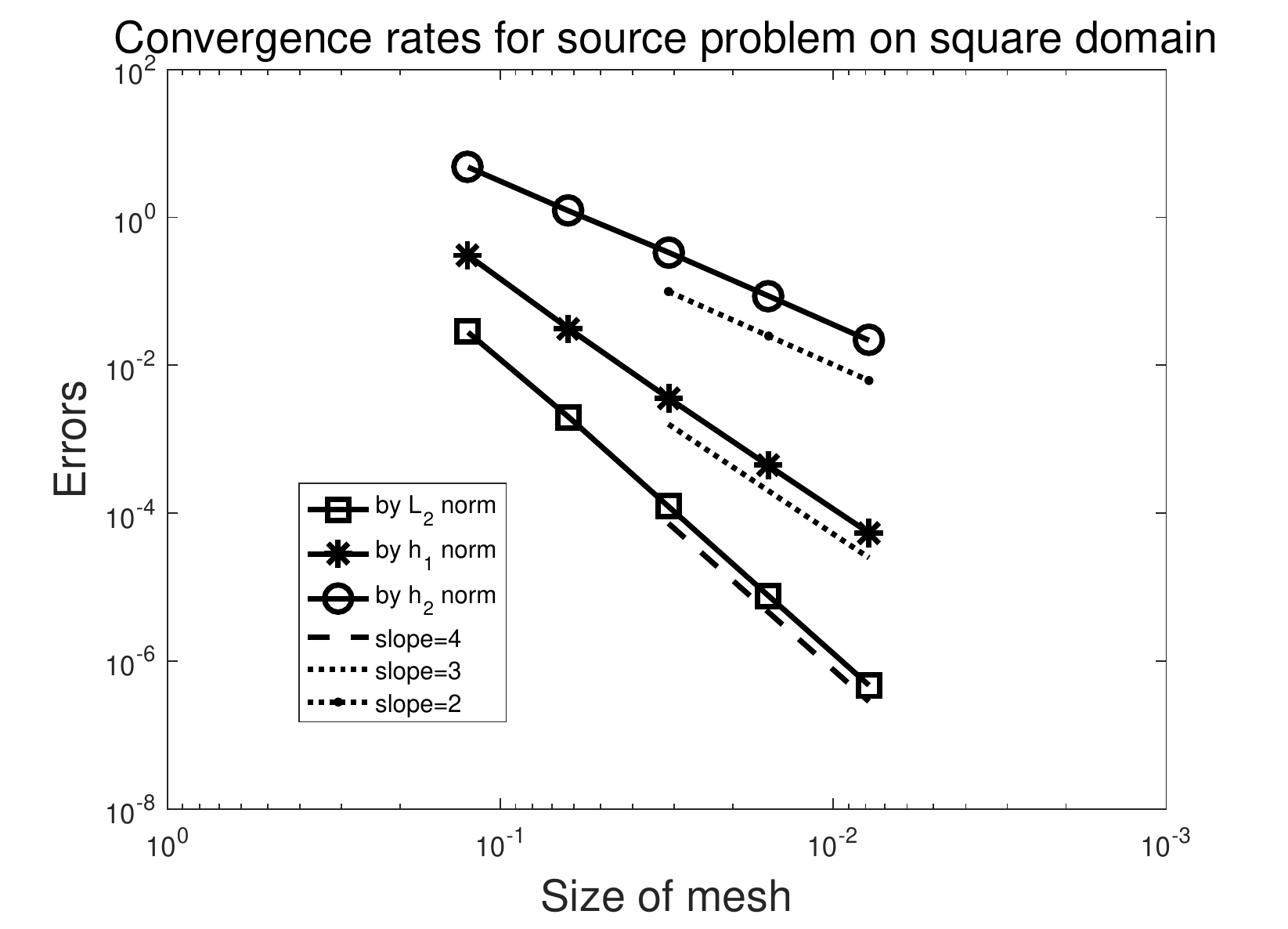}}
\subfigure{\includegraphics[width=0.45\textwidth,height=0.4\textwidth]{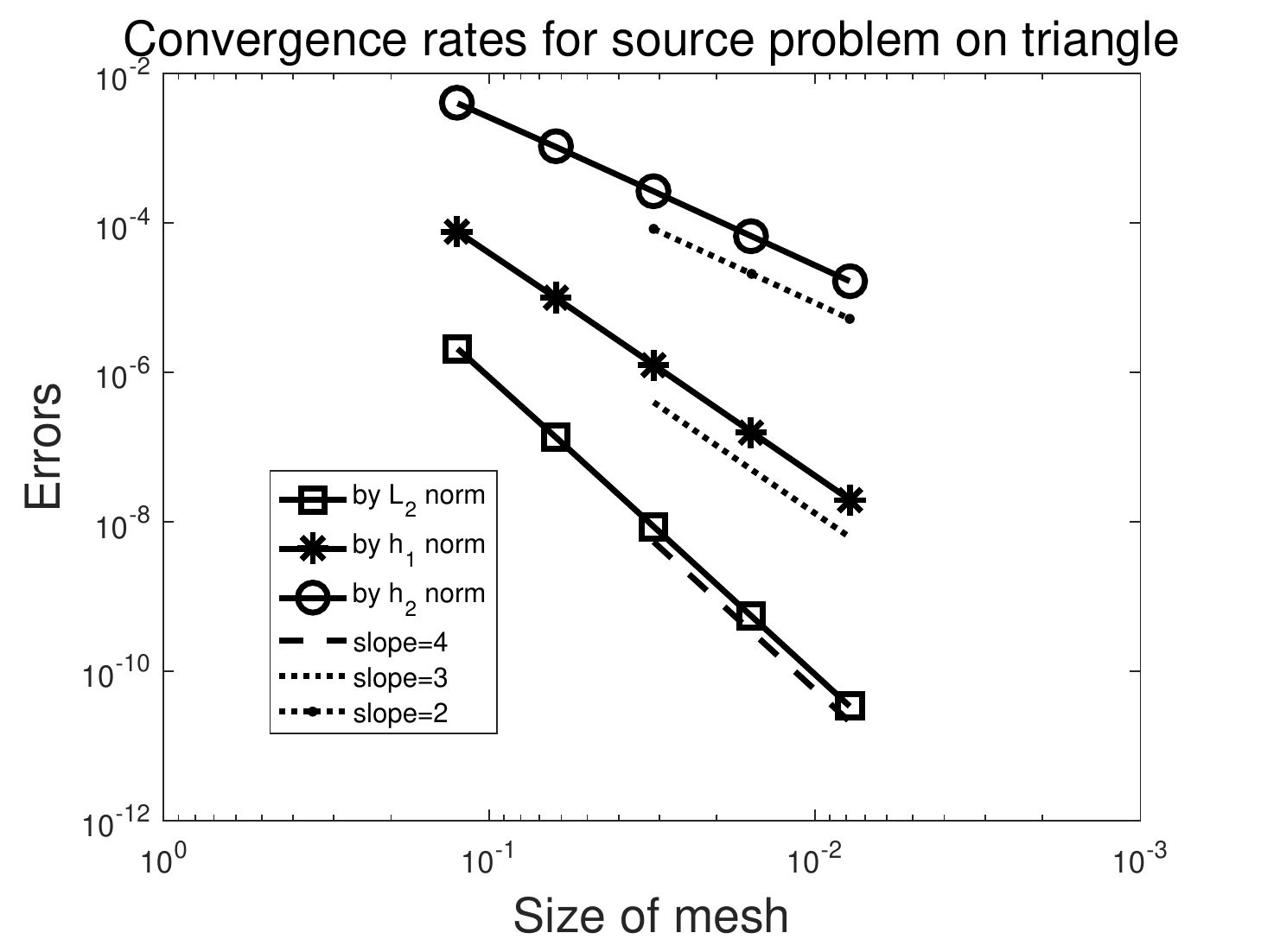}}
\caption{The convergence rate for the bi-elastic source problem. Left: for \textbf{Example 1}. Right: for \textbf{Example 2}.}
\label{BoundaryProblem}
\end{figure}

\subsection{The bi-elastic eigenvalue problem}

In this part, we consider the bi-elastic eigenvalue problem \eqref{ElasticEigenvalueProblem}.

\textbf{Example 3}: We test on the square domain $D_1=[0,1]^2$ and take the Lam\'{e} parameters $\lambda=\frac{1}{4}, \ \mu=\frac{1}{16}$, the constant coefficient $\beta=1$.

The lowest six computed eigenvalues on the successive five meshes are showed in Table \ref{tab:example3}. The following features can be from the numerical experiments.

(1) The convergence rate approximates fourth order which is consistent with the theoretical analysis.

(2) When the mesh size is small enough, the series of computed eigenvalues tends to monotonically decrease. Namely, a guaranteed upper bound of the eigenvalue can be expected.

\begin{table}[htp]
\caption{\label{tab:example3}The performance of {\bf $B_{h0}^3$}  for \textbf{Example 3}.}
\begin{tabular}{cccccccc}\hline
Mesh&1&2&3&4&5&Trend&$Ord_{\lambda}$\\\hline
$\lambda_1$&25.35774&23.39262&23.18043&23.16308&23.16188&$\searrow$&3.85629\\
$\lambda_2$&53.59356&50.42141&50.004164&49.968005&49.965453&$\searrow$&3.82492\\
$\lambda_3$&61.04122&51.06578&50.058280&49.972201&49.965760&$\searrow$&3.74039\\
$\lambda_4$&109.18534&105.40772&103.491568&103.206973&103.186660&$\searrow$&3.80846\\
$\lambda_5$&120.81714&106.74271&105.122452&105.092762&105.090460&$\searrow$&3.68859\\
$\lambda_{6}$&130.19532&106.74793&105.253200&105.104451&105.091409&$\searrow$&3.51167\\\hline
\end{tabular}
\end{table}

\textbf{Example 4}: We consider the bi-elastic eigenvalue problem on square domain $D_1=[0,1]^2$ with Lam\'{e} parameters $\lambda=\frac{1}{4}, \ \mu=\frac{1}{16}$ and the non-constant coefficient $\beta(\boldsymbol{x})=8+x_1-x_2$.

The convergence rate for the lowest six eigenvalues is showed in Table \ref{tab:example4}. The theoretical fourth-order convergence rate can be observed.

\begin{table}[htp]
\caption{\label{tab:example4}The performance of {\bf $B_{h0}^3$} for \textbf{Example 4}.}
\begin{tabular}{cccccccc}\hline
Mesh&1&2&3&4&5&Trend&$Ord_{\lambda}$\\\hline
$\lambda_1$&202.60084&186.85880&185.15778&185.01868&185.00907&$\searrow$&3.85612\\
$\lambda_2$&428.15681&402.73873&399.39743&399.10787&399.08743&$\searrow$&3.82485\\
$\lambda_3$&487.85256&407.93399&399.86251&399.17279&399.12120&$\searrow$&3.74081\\
$\lambda_4$&871.47709&841.01287&824.91149&822.72993&822.57169&$\searrow$&3.78513\\
$\lambda_5$&965.09121&847.47324&838.68117&838.43824&838.41936&$\searrow$&3.68622\\
$\lambda_{6}$&1041.68270&858.25905&842.82598&841.53746&841.42863&$\searrow$&3.56551\\\hline
\end{tabular}
\end{table}

\textbf{Example 5}: We consider the bi-elastic eigenvalue problem on the triangle domain whose vertices are $(0,0), (1,0), (\frac{1}{2},\frac{\sqrt{3}}{2})$ with Lam\'{e} parameters $\lambda=\frac{1}{4}, \ \mu=\frac{1}{4}$ and the non-constant coefficient $\beta=4+x_1^2+x_2^2$.

The numerical results are showed in Table \ref{tab:example5}.
The convergence rate for the lowest six eigenvalues is the optimal fourth order. And a guaranteed upper bound of the eigenvalue can be expected.

\begin{table}[htp]
\caption{\label{tab:example5}The performance of {\bf $B_{h0}^3$} for \textbf{Example 5}.}
\centering
{\small
\begin{tabular}{cccccccc}\hline
Mesh&1&2&3&4&5&Trend&$Ord_{\lambda}$\\\hline
$\lambda_1$&9158.98871&9080.38967&9074.44378&9074.03982&9074.01382&$\searrow$&3.95755\\
$\lambda_2$&9397.12583&9317.45937&9311.35816&9310.94206&9310.91526&$\searrow$&3.95656\\
$\lambda_3$&12853.77410&12686.55168&12669.66102&12668.40917&12668.32704&$\searrow$&3.92991\\
$\lambda_4$&32158.86909&31378.08662&31295.61981&31289.24927&31288.82506&$\searrow$&3.90858\\
$\lambda_5$&32498.29559&31743.16901&31664.35473&31658.28934&31657.88586&$\searrow$&3.91005\\
$\lambda_{6}$&40925.18544&39991.62636&39921.29040&39916.42396&39916.10683&$\searrow$&3.93972\\\hline
\end{tabular}
}
\end{table}
\subsection{Comparison with the Morley Element Scheme}\label{BiElasticMorley}
We check the Morley element scheme for the eigenvalue problem \eqref{ElasticEigenvalueProblem}. Since
$\left(\beta\nabla\cdot\sigma(\boldsymbol{w}),\nabla\cdot\sigma(\boldsymbol{\varphi})\right)$ is not coercive on Morley element space.
From \eqref{Elastic_Derivations}, we can obtain that
\begin{small}
\begin{eqnarray}
\Big(\nabla\cdot\sigma(\boldsymbol{w}),\nabla\cdot\sigma(\boldsymbol{\varphi})\Big)
=(\lambda^2+2\lambda\mu)\Big(\nabla\nabla\cdot\boldsymbol{w}, \nabla\nabla\cdot\boldsymbol{\varphi}\Big)+\mu^2(\Delta\boldsymbol{w},\Delta\boldsymbol{\varphi}). \nonumber
\end{eqnarray}
\end{small}
Hence, under the assumption that $0<\hat{\alpha}<1$, the higher order variational formulation item can be transfered into
\begin{footnotesize}
\begin{eqnarray}
\Big(\nabla\cdot\sigma(\boldsymbol{w}),\nabla\cdot\sigma(\boldsymbol{\varphi})\Big)&=&\Big(\nabla\cdot\sigma(\boldsymbol{w}),\nabla\cdot\sigma(\boldsymbol{\varphi})\Big)-\hat{\alpha}\mu^2(\nabla^2\boldsymbol{w},\nabla^2\boldsymbol{\varphi})+\hat{\alpha}\mu^2(\nabla^2\boldsymbol{w},\nabla^2\boldsymbol{\varphi})\nonumber \\
&=&(1-\hat{\alpha})\Big(\nabla\cdot\sigma(\boldsymbol{w}),\nabla\cdot\sigma(\boldsymbol{\varphi})\Big)+\hat{\alpha}\Big(\nabla\cdot\sigma(\boldsymbol{w}),\nabla\cdot\sigma(\boldsymbol{\varphi})\Big) -\hat{\alpha}\mu^2(\nabla^2\boldsymbol{w},\nabla^2\boldsymbol{\varphi})+\hat{\alpha}\mu^2(\nabla^2\boldsymbol{w},\nabla^2\boldsymbol{\varphi}) \nonumber \\
&=&(1-\hat{\alpha})\Big(\nabla\cdot\sigma(\boldsymbol{w}),\nabla\cdot\sigma(\boldsymbol{\varphi})\Big)+\hat{\alpha}\mu^2(\Delta\boldsymbol{w},\Delta\boldsymbol{\varphi})+\hat{\alpha}(\lambda^2+2\lambda\mu)(\nabla\nabla\cdot\boldsymbol{w},\nabla\nabla\cdot\boldsymbol{\varphi})-\hat{\alpha}\mu^2(\nabla^2\boldsymbol{w},\nabla^2\boldsymbol{\varphi})+\hat{\alpha}\mu^2(\nabla^2\boldsymbol{w},\nabla^2\boldsymbol{\varphi})\nonumber\\
&=&(1-\hat{\alpha})\Big(\nabla\cdot\sigma(\boldsymbol{w}),\nabla\cdot\sigma(\boldsymbol{\varphi})\Big)+\hat{\alpha}(\lambda^2+2\lambda\mu)(\nabla\nabla\cdot\boldsymbol{w},\nabla\nabla\cdot\boldsymbol{\varphi})+\hat{\alpha}\mu^2(\nabla^2\boldsymbol{w},\nabla^2\boldsymbol{\varphi}),\nonumber
\end{eqnarray}
\end{footnotesize}
where
\begin{equation}
(\nabla^2 \boldsymbol{w},\nabla^2 \boldsymbol{\varphi})=\int_\Omega\sum_{s,t=1}^2\frac{\partial^2 w_1}{\partial x_s\partial x_t}\frac{\partial^2 \varphi_1}{\partial x_s\partial x_t}dx+\int_\Omega\sum_{s,t=1}^2\frac{\partial^2 w_2}{\partial x_s\partial x_t}\frac{\partial^2 \varphi_2}{\partial x_s\partial x_t}dx.  \nonumber
\end{equation}
For Morley element, we consider the following variational
formulation: find $(K,\boldsymbol{w})\in \mathcal{C}\times V$, for $\forall \boldsymbol{\varphi}\in V$, we have
{\footnotesize
\begin{equation}\label{NC_alphas_VariationalFormulation}
\Big(\beta(x)\nabla\cdot\sigma(\boldsymbol{w}),\nabla\cdot\sigma(\boldsymbol{\varphi})\Big)=\Big((\beta(x)-\tilde{\alpha})\nabla\cdot\sigma(\boldsymbol{w}),\nabla\cdot\sigma(\boldsymbol{\varphi})\Big)+\tilde{\alpha}\mu^2(\nabla^2 \boldsymbol{w},\nabla^2\boldsymbol{\varphi})+\tilde{\alpha}(\lambda^2+2\lambda\mu)\Big(\nabla\nabla\cdot\boldsymbol{w},\nabla\nabla\cdot\boldsymbol{\varphi}\Big)=K(\boldsymbol{w},\boldsymbol{\varphi}).  
\end{equation}
}
where $0<\tilde{\alpha}<\beta_{min}$ is constant.

The items in the middle of (\ref{NC_alphas_VariationalFormulation}) are used to guarantee the coercivity of variational problem.

The Morley element discretization space for $H_0^2(\Omega)$ is denoted by $V_h^M$. The corresponding discretized variational formulation is: Find $\boldsymbol{w_h}\in (V_h^M)^2$ and $K_h\in \mathcal{C}$, such that
{\footnotesize
\begin{equation}\label{Morley_alphas_VariationalFormulation}
\Big((\beta(x)-\alpha)\nabla\cdot\sigma(\boldsymbol{w_h}),\nabla\cdot\sigma(\boldsymbol{\varphi_h})\Big)+\alpha\mu^2(\nabla^2 \boldsymbol{w_h},\nabla^2\boldsymbol{\varphi_h})+\alpha(\lambda^2+2\lambda\mu)\Big(\nabla\nabla\cdot\boldsymbol{w_h},\nabla\nabla\cdot\boldsymbol{\varphi_h}\Big)=K_h(\boldsymbol{w_h},\boldsymbol{\varphi_h}),\ \ \forall \boldsymbol{\varphi_h}\in (V_h^M)^2,
\end{equation}
}
where $0<\alpha<\beta_{min}$ is a constant.
We test the Morley element method on \textbf{Example 4} and \textbf{Example 5}. For \textbf{Example 4}, the lowest ten computed eigenvalues on three successive grid levels are showed in the left figure of Figure \ref{BiElasticProblem}.
We can observe that the numerical results are sensitive to the parameter $\alpha$.
For \textbf{Example 5}, the numerical results are showed in the right figure of Figure \ref{BiElasticProblem}. For different parameter $\alpha$, the computed eigenvalues are different. For different $\beta(x)$, the optimal $\alpha$ is also different.
\begin{figure}
\centering
\subfigure{\includegraphics[width=0.45\textwidth,height=0.4\textwidth]{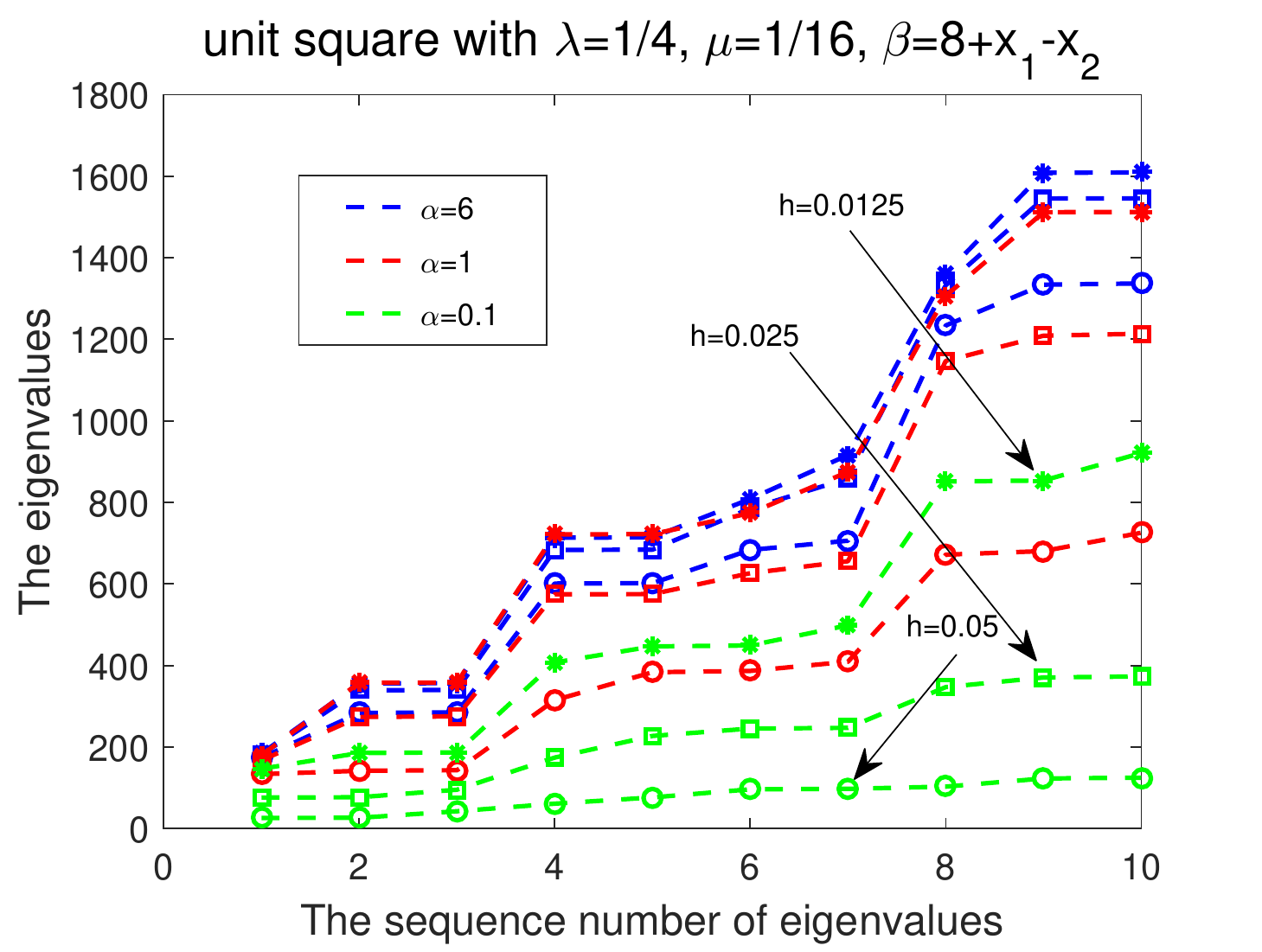}}
\subfigure{\includegraphics[width=0.45\textwidth,height=0.4\textwidth]{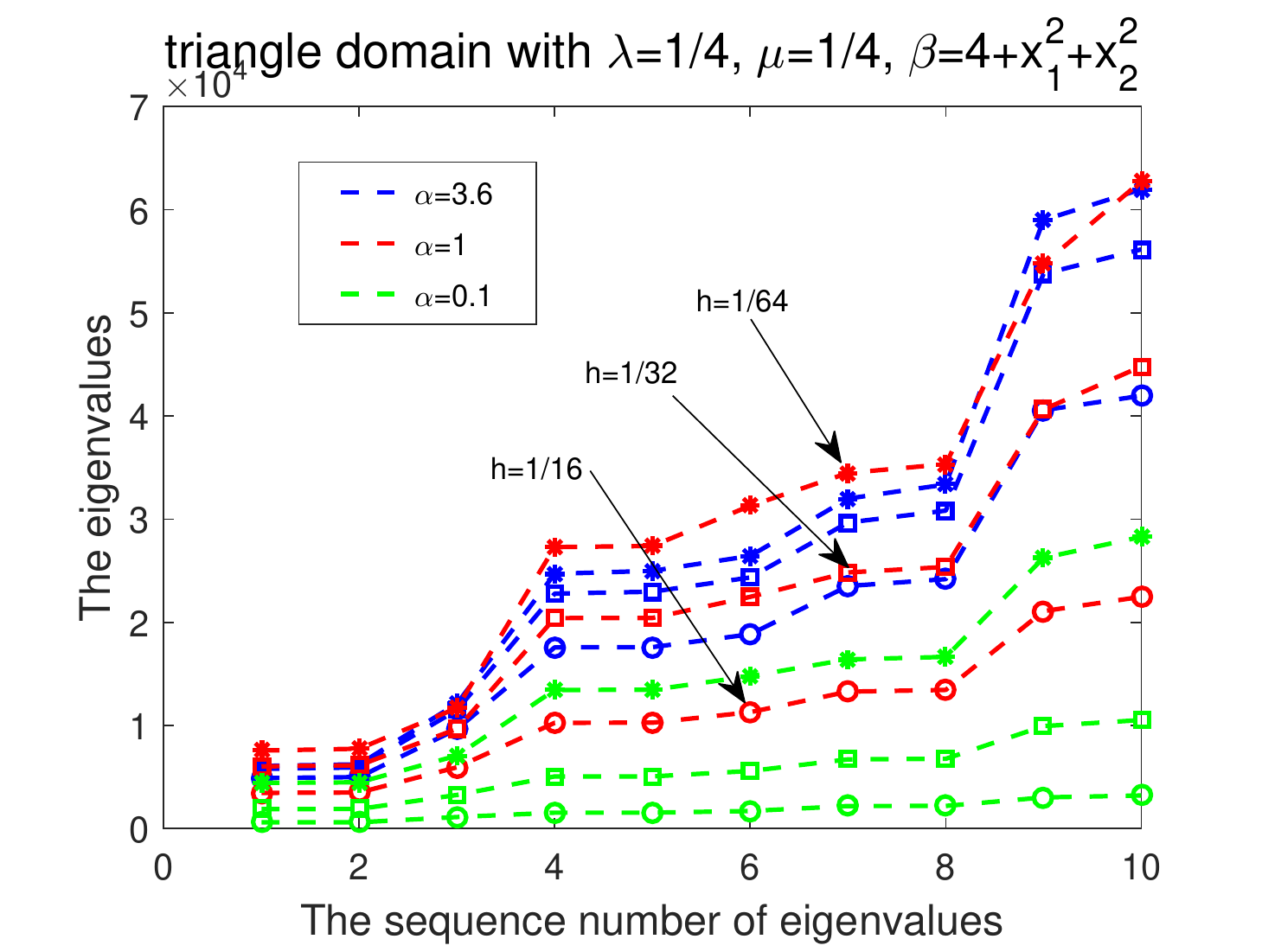}}
\caption{The numerical performance by Morley element for bi-Elastic eigenvalue problem with non-constant coefficient. Left: for $\lambda=\frac{1}{4},\mu=\frac{1}{16}$ with $\beta(x)=8+x_1-x_2$ on unit square; Right: for $\lambda=\mu=\frac{1}{4}$ with $\beta(x)=4+x_1^2+x_2^2$ on triangle domian.}
\label{BiElasticProblem}
\end{figure}

\subsection{The elastic transmission eigenvalue problem}
Here we focus on the
case $\rho_0(x)\leq Q\leq1\leq q_*\leq\rho_1(x)$. For the case $\rho_1(x)\leq Q_*\leq1\leq q\leq\rho_0(x)$, it follows analogously. We present some numerical results using three domains:
the unit square $\Omega_1=[0,1]^2$, a triangle domain $\Omega_2$ whose vertices are $(0,0), (1,0), (\frac{1}{2},\frac{\sqrt{3}}{2})$ and an L-shaped domain given by $\Omega_3=(0,1)\times (0,1) \setminus [1/2, 1]\times[1/2,1]$. The initial mesh size is chosen as $h_1=\frac{1}{2}$.
Five levels of uniformly refined triangular meshes are generated for numerical experiments and $h_k=h_{k-1}/2 (k=2,3,4,5)$.
The eigenvalues computed on the mesh of size $h_k$ are denoted by $\lambda_{h_k} (k=1,2,3,4,5)$.
The convergent orders are computed by
\begin{equation}
log_2(|\frac{\lambda_{h_l}-\lambda_{h_{5}}}{\lambda_{h_{l+1}}-\lambda_{h_{5}}}|),\ \ \ \ l=1,2,3.
\end{equation}

We consider the following examples.

\textbf{Example 6.} The unit square domain $\Omega_1$ with Lam\'{e} parameters $\mu = 1/4,\ \lambda = 1/4$ and the mass density $\rho_0(x) = \frac{1}{20}, \ \rho_1(x) = 3$.

\textbf{Example 7.} The unit square domain $\Omega_1$ with Lam\'{e} parameters $\mu = 1/4$, $\lambda = 1/12$ and the mass density $\rho_0(x) = \frac{1}{2}$, $\rho_1(x) = 4+x_1-x_2$.

\textbf{Example 8.} The triangle domain $\Omega_2$ with Lam\'{e} parameters $\mu = 1/4,\ \lambda = 1/16$ and the mass density $\rho_0(x) = \frac{1}{8},\  \rho_1(x) = 4+x_1^2+x_2^2$.

\textbf{Example 9.} The L-shaped domain $\Omega_3$ with Lam\'{e} parameters $\mu = 1/4,\ \lambda = 1/16$ and the mass density $\rho_0(x) = 1, \ \rho_1(x) = 4$.

Table~\ref{tablefirstten} gives the first ten eigenvalues of the above four examples, which is consistent with the result of $\omega^2$ in \cite{XiJi2018}. For \textbf{Example 6}, \textbf{Example 7}, \textbf{Example 8}, the finest degrees of freedom are 23564. The convergence order of the lowest
six eigenvalues are showed in Figure~\ref{ETEP_6evs_example6}, Figure~\ref{ETEP_6evs_example7}, Figure~\ref{ETEP_6evs_example8}, respectively. The convergence rate on unit square or triangle domain is approximately 4 which is optimal and consistent with the theoretical analysis. For \textbf{Example 9}, the finest degrees of freedom are 47628.
The convergence rate on L-shaped domain is showed in Figure~\ref{ETEP_6evs_example9} which is lower than 4 due to the fact that reentrant corner leads to the low regularity of eigenfunctions which is consistent with the results in \cite{JiLiSun, XiJi2018}.

\begin{table}
\begin{center}
\begin{tabular}{ccccc}
\hline
Eigenvalue & \textbf{Example 6} &\textbf{Example 7}  &\textbf{Example 8} &\textbf{Example 9}    \\
\hline
$\Lambda_1$ &8.064689    &2.172958      &3.992401&  3.612558-3.041481i   \\
$\Lambda_2$ &9.561642    &3.122101      &5.491747&  3.612558+3.041481i   \\
$\Lambda_3$ &9.561852    &3.125644      &5.586717&  4.870908  \\
$\Lambda_4$ &14.001823   &4.052862      &7.522283&  5.284471 \\
$\Lambda_5$ &14.002024   &4.574646      &7.619984&  6.293109   \\
$\Lambda_6$ &14.256426   &5.469985      &8.202532&  6.654243  \\
$\Lambda_7$ &15.125605   &5.558520      &9.379437&  7.394237  \\
$\Lambda_8$ &20.589411   &6.065146      &9.436227&  5.835047-4.720085i  \\
$\Lambda_9$ &20.590683   &6.918999      &10.469256& 5.835047+4.720085i   \\
$\Lambda_{10}$&20.762326 &7.124076      &11.046607& 8.020355   \\
\hline
\end{tabular}
\end{center}
\caption{The first ten elastic transmission eigenvalues on the finest mesh level of $h\approx0.015625$.}
\label{tablefirstten}
\end{table}

\begin{figure}[ht]
\begin{center}
\includegraphics[scale=0.55]{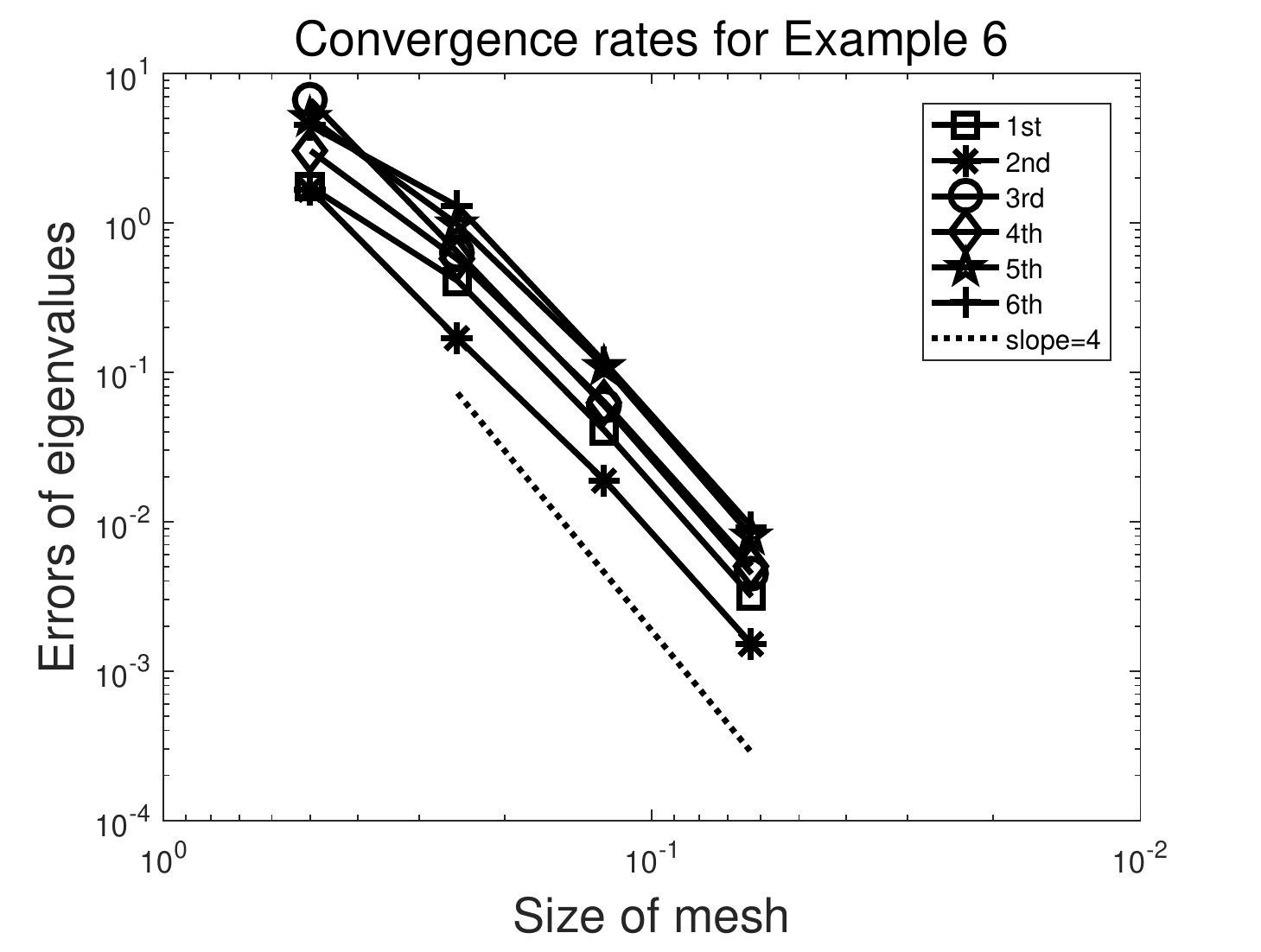}
\end{center}
\caption{The convergence rates for the lowest six real eigenvalues for \textbf{Example 6} by $B_{h0}^3$. Y-axis means
the numerical error of eigenvalues; X-axis means the size of mesh.}
\label{ETEP_6evs_example6}
\end{figure}

\begin{figure}[ht]
\begin{center}
\includegraphics[scale=0.55]{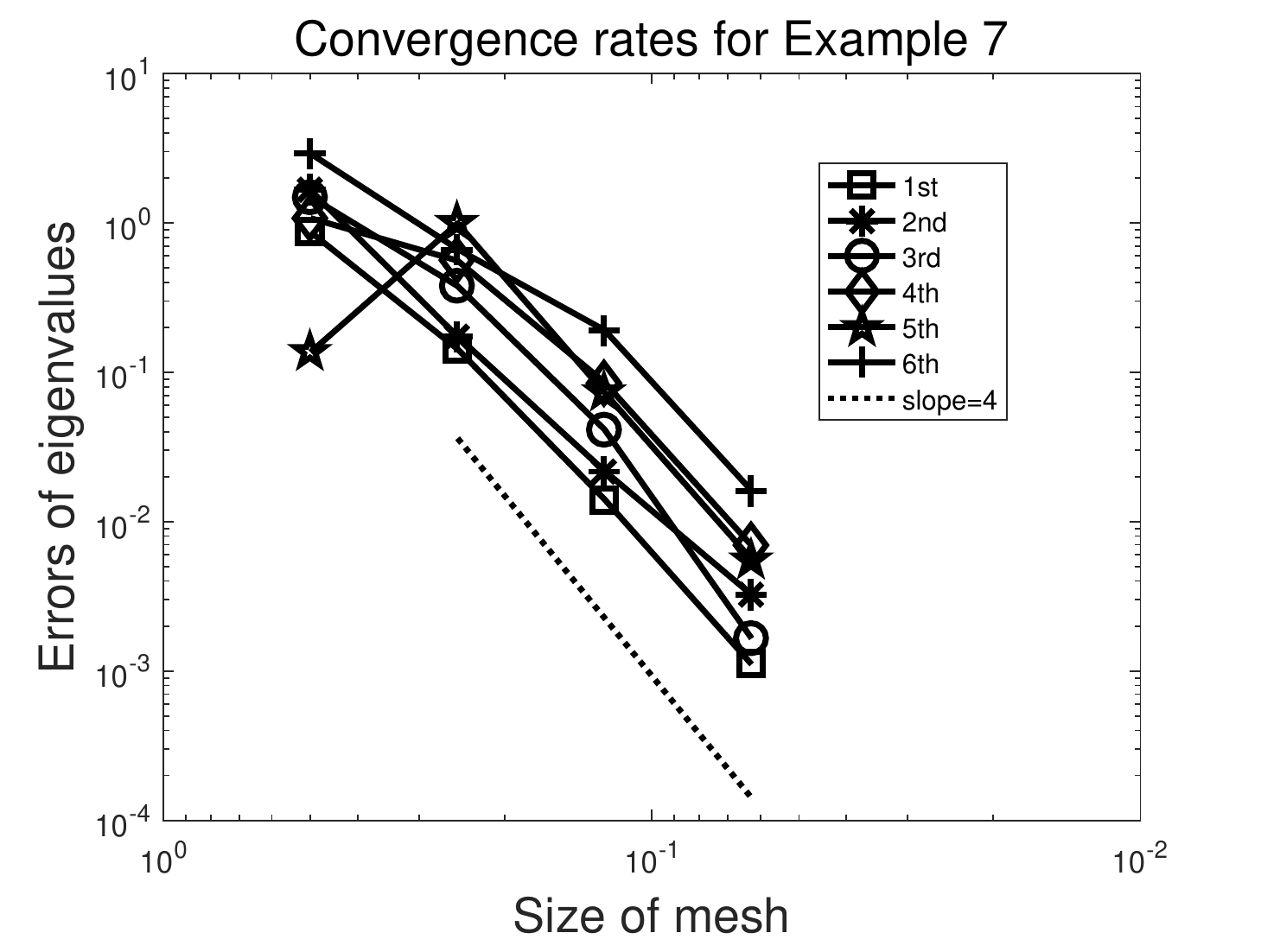}
\end{center}
\caption{The convergence rates for the lowest six real eigenvalues for \textbf{Example 7} by $B_{h0}^3$. Y-axis means
the numerical error of eigenvalues; X-axis means the size of mesh.}
\label{ETEP_6evs_example7}
\end{figure}

\begin{figure}[ht]
\begin{center}
\includegraphics[scale=0.55]{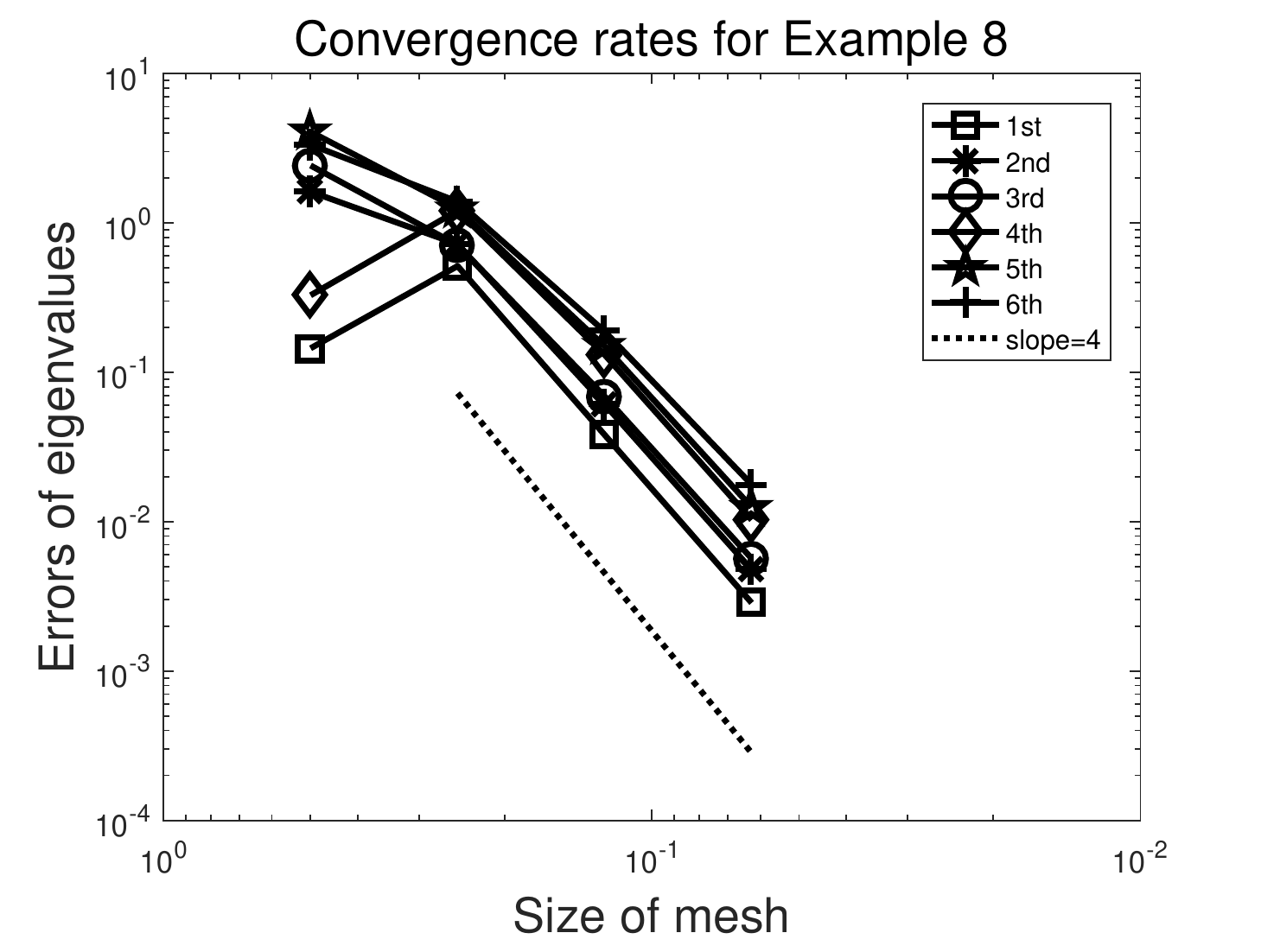}
\end{center}
\caption{The convergence rates for the lowest six real eigenvalues for \textbf{Example 8} by $B_{h0}^3$. Y-axis means
the numerical error of eigenvalues; X-axis means the size of mesh.}
\label{ETEP_6evs_example8}
\end{figure}

\begin{figure}[ht]
\begin{center}
\includegraphics[scale=0.55]{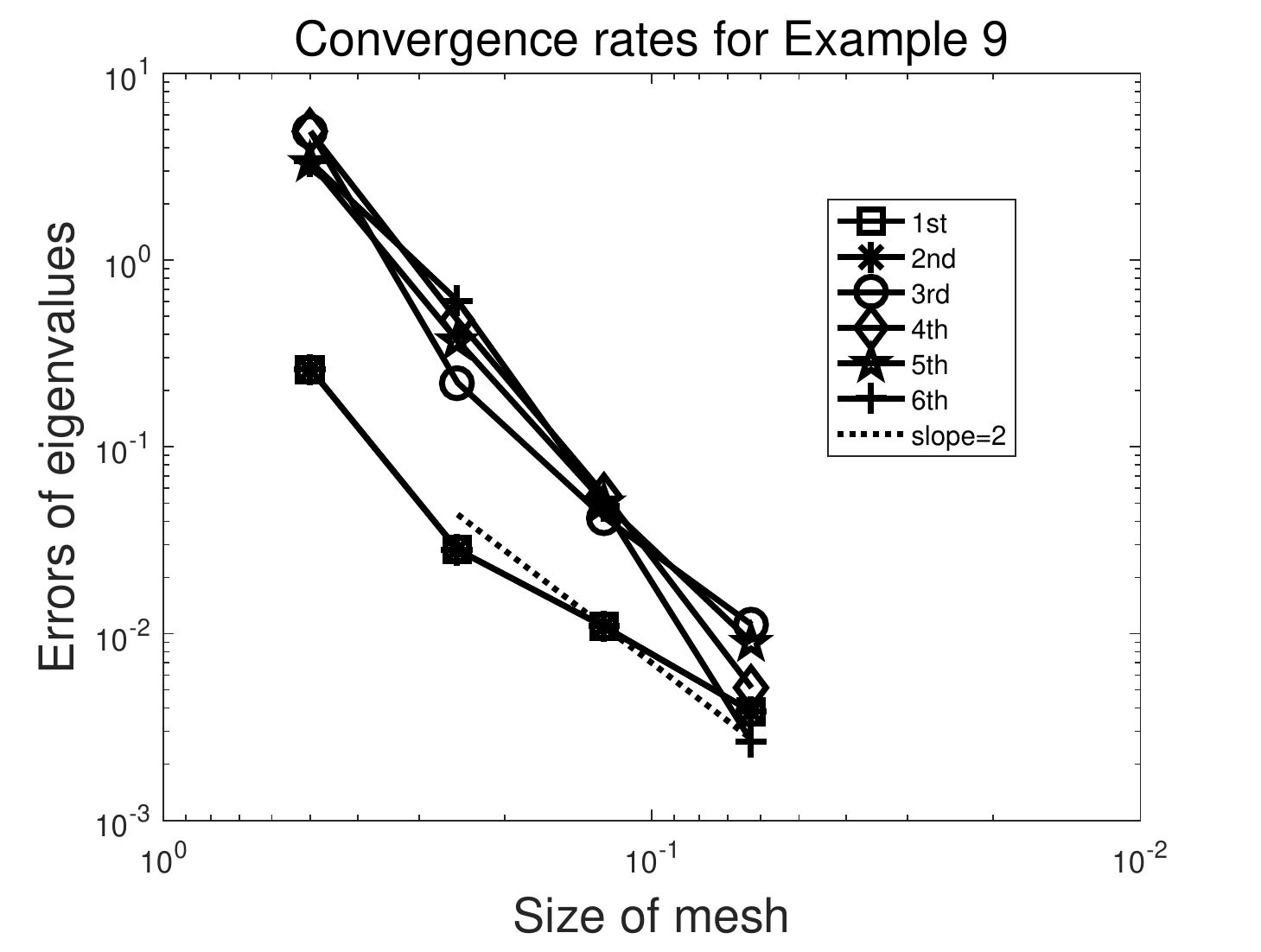}
\end{center}
\caption{The convergence rates for the lowest six real eigenvalues for \textbf{Example 9} by $B_{h0}^3$. Y-axis means
the numerical error of eigenvalues; X-axis means the size of mesh.}
\label{ETEP_6evs_example9}
\end{figure}

\subsection{Comparision with the Morley scheme}
In this subsection, we check the Morley element scheme for the elastic transmission eigenvalue problem \eqref{WeakFormula}. We take the case $\rho_0(x)\leq Q\leq1\leq q_*\leq\rho_1(x)$ for illustration. For the case $\rho_1(x)\leq Q_*\leq1\leq q\leq\rho_0(x)$, it can be treated analogously.
For Morley element, we consider the expansion of varitional formulation \eqref{WeakFormula}: find $\tau$ and $\boldsymbol{w}\neq0\in V$ such that, for $\forall\boldsymbol{\varphi}\in V$
\begin{equation}
\left((\rho_1-\rho_0)^{-1}\nabla\cdot\sigma(\boldsymbol{w}),\nabla\cdot\sigma(\boldsymbol{\varphi})\right)+\tau\left(\frac{\rho_0}{\rho_1-\rho_0}\boldsymbol{w},\nabla\cdot\sigma(\boldsymbol{\varphi})\right)+
\tau\left(\frac{\rho_1}{\rho_1-\rho_0}\nabla\cdot\sigma(\boldsymbol{w}),\boldsymbol{\varphi}\right)+\tau^2\left(\frac{\rho_0\rho_1}{\rho_1-\rho_0}\boldsymbol{w},\boldsymbol{\varphi}\right)=\boldsymbol{0}.
\end{equation}
In order to guarantee the coerciveness on the Morley finite element space, for the non-constant coefficient, we assume $0<\alpha\leq\rho_{min}\triangleq\min\frac{1}{\rho_1(x)-\rho_0(x)}$ and transform the higher order variational formulation to the following form:
{\small
\begin{equation}\label{revised_bilinear_form}
\Big((\rho_1-\rho_0)^{-1}\nabla\cdot\sigma(\boldsymbol{w}),\nabla\cdot\sigma(\boldsymbol{\varphi})\Big)=\Big(((\rho_1-\rho_0)^{-1}-\alpha)\nabla\cdot\sigma(\boldsymbol{w}),\nabla\cdot\sigma(\boldsymbol{\varphi})\Big)+\alpha\mu^2(\nabla^2 \boldsymbol{w},\nabla^2\boldsymbol{\varphi})+\alpha(\lambda^2+2\lambda\mu)\Big(\nabla\nabla\cdot\boldsymbol{w},\nabla\nabla\cdot\boldsymbol{\varphi}\Big).  
\end{equation}
}
The form on the right-hand side of (\ref{revised_bilinear_form}) guarantees the coercivity of the variational formulation on $V$. However, in the practical computation, the numerical performance is sensitive to the choice of $\alpha$. The left figure in Figure \ref{ElasticTEP} shows the numerical performance by Morley element for unit square domain $\Omega_1=[0,1]^2$ with the cofficients
$\lambda=1/4,\ \mu= 1/4,\ \rho_0(x)=\frac{1}{20},\ \rho_1(x)=3$. For a fixed $\alpha$, we present the lowest 10 computed eigenvalues on three successive grid levels. It's observed that the numerical results are greatly dependent on the choice of $\alpha$.
The right figure in Figure \ref{ElasticTEP} shows the numerical performance for unit square domain with $\lambda=1/4,\mu= 1/12,\rho_0(x)=\frac{1}{2},\rho_1(x)=4+x_1-x_2$. For different index of refractions, the optimal choice of $\alpha$ is also different.

\begin{figure}
\centering
\subfigure{\includegraphics[width=0.45\textwidth,height=0.4\textwidth]{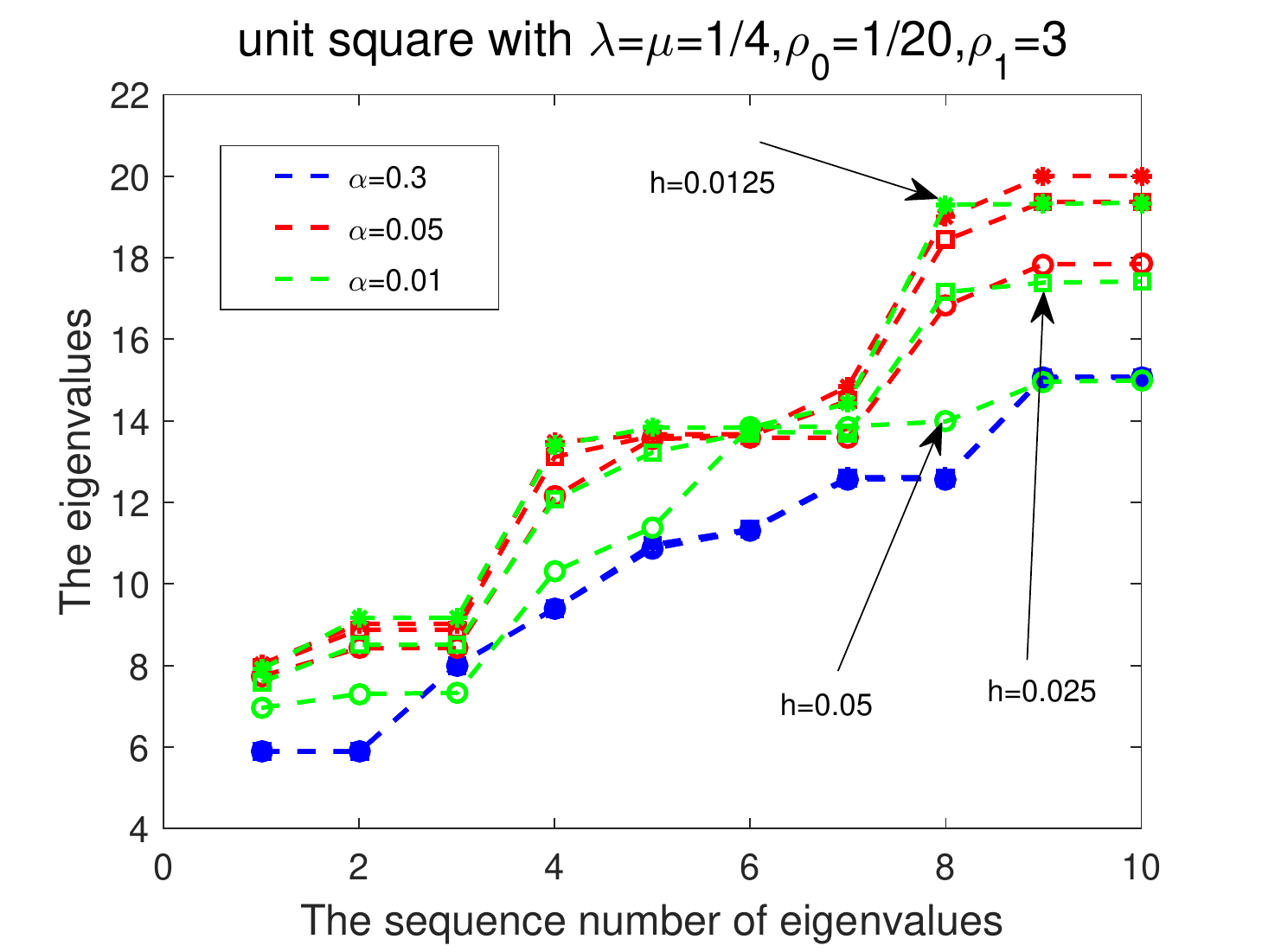}}
\subfigure{\includegraphics[width=0.45\textwidth,height=0.4\textwidth]{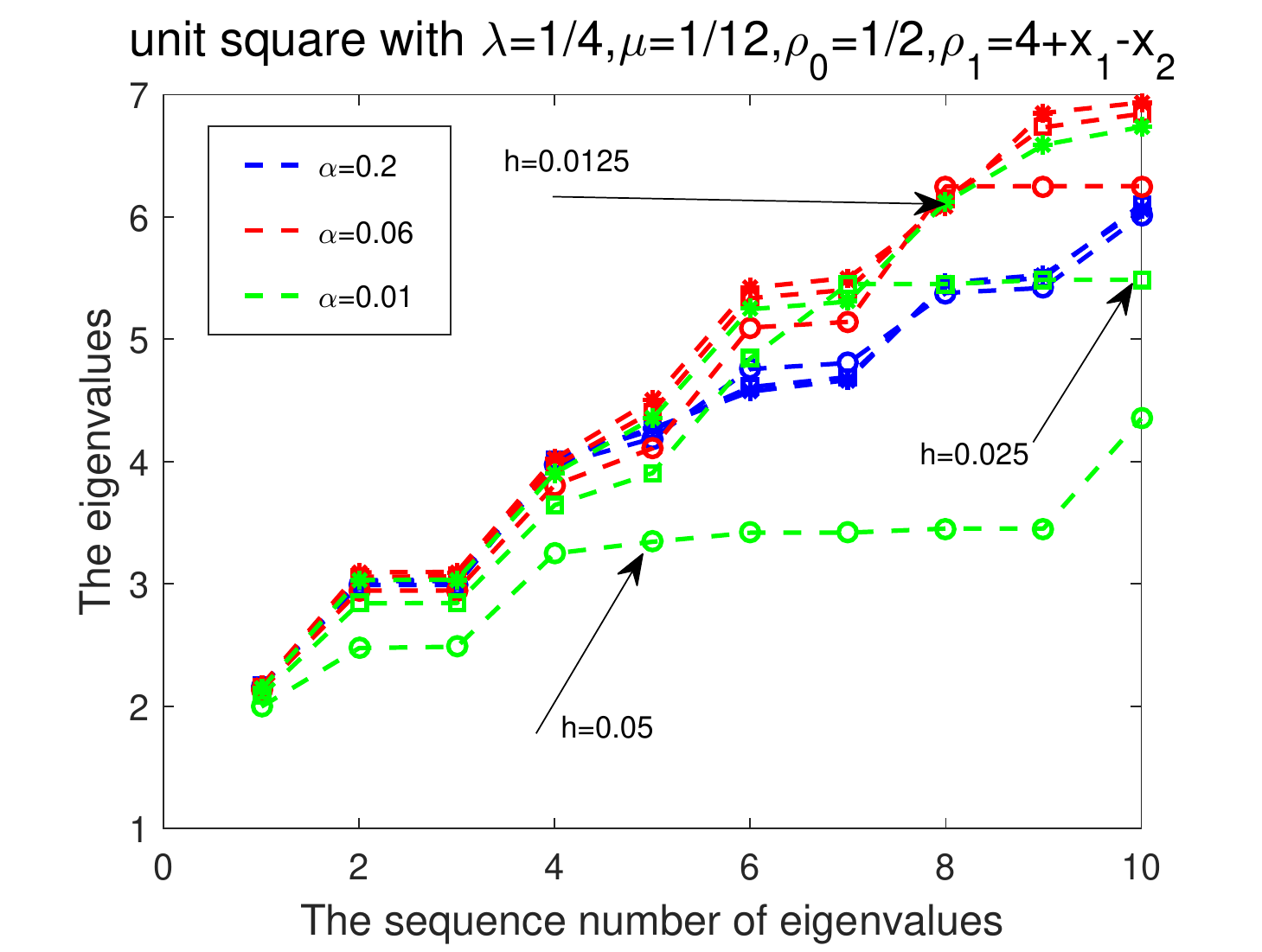}}
\caption{Morley element scheme for elastic transmission eigenvalue problem. Left:  $\lambda=\mu=\frac{1}{4}$ with $\rho_0(x)=\frac{1}{20}$, $\rho_1(x)=3$ on unit square; Right: $\lambda=\frac{1}{4},\mu=\frac{1}{12}$ with $\rho_0(x)=\frac{1}{2}$, $\rho_1(x)=4+x_1-x_2$ on unit square.}
\label{ElasticTEP}
\end{figure}

\end{document}